\documentclass[11pt]{amsart}
\usepackage{amssymb}
\usepackage[
            colorlinks=true,
            urlcolor=blue,       
            filecolor=blue,     
            linkcolor=blue,       
            citecolor=blue,         
            pdftitle= {SE(3) invariants for surfaces]},
            pdfauthor={Y. Hayut and D. Lehavi},
            pdfsubject={},
            pdfkeywords={}
            pagebackref,
            pdfpagemode=UseNone,
            bookmarksopen=true]
            {hyperref}
\usepackage{hyperref}
\usepackage{color}
\usepackage{mathtools}
\usepackage{tikz-cd}
\newcommand{\BC}{{\mathbb{C}}}
\newcommand{\BZ}{{\mathbb{Z}}}
\newcommand{\BF}{{\mathbb{F}}}
\newcommand{\BN}{{\mathbb{N}}}
\newcommand{\BR}{{\mathbb{R}}}
\newcommand{\BS}{{\mathbb{S}}}
\newcommand{\gD}{\Delta}
\newcommand{\gd}{\delta}
\newcommand{\gb}{\beta}
\newcommand{\gc}{\gamma}

\newcommand{\gep}{\epsilon}
\newcommand{\ga}{\alpha}

\newcommand{\cM}{{\mathcal{M}}}

\newcommand{\caD}{{\mathcal{D}}}

\newcommand{\RSS}{{W_7}}
\DeclareMathOperator{\Ad}{Ad}
\DeclareMathOperator{\range}{range}
\DeclareMathOperator{\rank}{rank}
\newcommand{\ti}[1]{\tilde{#1}}
\newcommand{\ol}[1]{\overline{#1}}
\newcommand{\Pol}{\mathrm{Pol}}

\newcommand{\spn}{\mathrm{span}}

\newcommand{\sm}{\smallsetminus}
\swapnumbers
\theoremstyle{plain}
\newtheorem{lma}{Lemma}[section]
\newtheorem{thm}[lma]{Theorem}

\newtheorem*{MainThm}{Main Theorem}
\newtheorem{prp}[lma]{Proposition}
\newtheorem{cor}[lma]{Corollary}
\theoremstyle{definition}
\newtheorem{prd}[lma]{Proposition-Definition}
\newtheorem{dfn}[lma]{Definition}
\newtheorem{rmr}[lma]{Remark}
\newtheorem{ntt}[lma]{Notation}
\newtheorem{exm}[lma]{Example}
\newtheorem{dsc}[lma]{}
\newtheorem{question}[lma]{Question}

\begin{document}
\title[$SE(3)$ invariants for surfaces]{Complete $SE(3)$ invariants for a comeagre set of $C^3$ compact orientable surfaces in $\BR^3$}
\begin{abstract}
  We introduce invariants for compact $C^1$-orientable surfaces (with boundary)
  in $\BR^3$ up to rigid transformations. Our invariants are certain degree
  four polynomials in the moments of
  the delta function of the surface. We give an effective and numerically stable
  inversion algorithm for retrieving the surface from the invariants,
  which works on a comeagre subset of $C^3$-surfaces.
\end{abstract}
\author{Yair Hayut}
\email{yair.hayut@mail.huji.ac.il}
\author{David Lehavi}
\email{dlehavi@gmail.com}
\date{\today}
\maketitle
%
\section{Introduction}
%
One of the classical problems in three dimensional computational geometry is
finding invariants of surfaces in $\BR^3$ under the group of rigid orientation
preserving affine transformations $SE(3)$, which
distinguish between non-equivalent surfaces, and are continuous with respect to
some reasonable topology.

This paper tackles this problem by borrowing ideas from geometric
invariant theory (G.I.T.). Loosely speaking, we do the following:
\begin{enumerate}
\item Find an ambient space where $C^1$-orientable surfaces in $\BR^3$ naturally reside.
  Our ambient space is the space of compactly supported tempered distributions
  on $\BR^3\times \BS^2$, where each surface is associated
  with its delta function (taking the surface and its normal). We identify
  these delta functions with their respective generalized Fourier coefficients.
\item Take the convolution of the delta function
  with its ``mirror images''; first along $\BR^3$, and then along $SO(3)$.
\item\label{i:bound}
  Bound the set of moduli points over which a surface may not be
  recovered from these doubly convolved distributions (the bad points).
\end{enumerate}
The first main result of this paper is an explicit description (in
definitions \ref{d:property_star} and \ref{d:property_star_star}) of
a bound on the set of bad points, for $C^2$-surfaces. We show that
this bounding set is relatively meagre when restricted to $C^3$-surfaces.
We refer the
reader to Section \ref{S:results} for the precise definitions as well as the
statement of this result.

Moreover, we show that
the generalized Fourier coefficients of these distributions
are well behaved, and use this nice behavior to give the second main result:
{\em An effective inversion algorithm over a generic point}, i.e.\ we recover the
surface (away from a comeagre set, and up to the action of $SE(3)$) from
the coefficients of the doubly convolved distribution.

\subsection*{Related work}\ \\
The problem of finding invariants to surfaces in $\BR^3$ is studied in three
different disciplines: Mathematics,
Chemistry (as a method to reproduce molecular structure given
Cryogenic Electron Microscopy)
and Computer Science (specifically in Computer Vision).
The approaches employed were also rather diverse.

As we indicated in the introduction, this work is motivated by, but not exactly analogous to G.I.T.
Indeed, in G.I.T.\ the general plan of attack is:
\begin{enumerate}
\item Find an ambient space.
\item Find elements for which the group action is ``problematic'', and discard
  them.
\item Compute the quotient.
\end{enumerate}
On the one hand our group is nicer than the general linear groups,
which typically appear in G.I.T.: by Schwartz (see \cite{Sc})
real orthogonal groups act stably, so instability (in the G.I.T. sense) may
only come from
translations.
On the other hand,
we define elements as ``problematic'' in a much harsher way
for two reasons: The
first is that looking at the Fourier
coefficients of the double convolution
is the analogue of only looking at certain degree 4 polynomials in the
coefficients inside the ring
of invariants, so our invariants are limited.
The second is that we consider general surfaces, and not
merely algebraic ones. Moreover, in
order to get effective results, we take successive approximations.

From a classical standpoint, one may approach the problem
by endowing the space of surfaces with some topology, and construct a continuous
(generically) one to one map to some normed space.
The algebraic side of the ideas we work with here ---
of using the Fourier coefficients ---
originated in \cite{Kam}. In \cite{SKKLS}, the authors
presented an effective algorithm for reconstructing functions in
$L^2(\BS^2)$ up to $SO(3)$-action from the third degree invariants in their
Fourier decomposition. Similar ideas were rediscovered, although
not as thoroughly developed, in the Computer Vision community;
Specifically see \cite{KFR} (building on earlier work in \cite{Horn}). In all
these papers, the authors solved a ``sub-problem'' of the original problem, by
mapping the space of surface to $L^2(\BS^2)$, either by
using the Gauss Map, or by projection from some distinguished point;
both maps which are very far from being injective. The two main differences
between our work and the strongest previous
result along this line --- \cite{SKKLS} -- are that
our arguments are geometric and not algebraic, and that
we work with $SE(3)\cong\BR^3\rtimes SO(3)$ instead of $SO(3)$; this
comes at the cost of having a generic solution.

Even in this setting of constructing a moduli problem as a set of moduli points,
there are alternatives to the path we took:
E.g.\ in \cite{Ol}, the author develops an algebra of differential forms
for each surface, and in \cite{GHP}, the authors use the classical
{\em slice method}
of invariant theory (for $SO(3)$ only though): solving the invariant problem
assuming that the
mass distribution of the surface has three distinct eigenvalues, and solving
the moduli problem for a finite group. In a recent work, \cite{KRV}, Kogan, Ruddy, and Vinzant solved the corresponding problem for algebraic curves (over $\mathbb{C}^2$)
for the action of the projective group and its subgroups, using tools from algebraic geometry.

From a more modern perspective, a ``correct'' moduli problem is ``What is a
{\em good family} of algebraic varieties? Can we describe all good families in
an {\em optimal} manner?'' In the setting of differentiable surfaces, this
approach seems to still be at its infancy ---
see \cite{FOOO,Jo,MW}.

Finally, in the applied computer science community there is a huge number of
attacks on the general problem of ``finding good
$SE(3)$-invariants of a surface''. The most prevalent are applying
Deep Learning to images of projections on of the surfaces, some discretizations
of the surface, or even spherical harmonic coefficients of the surface in
some fixed coordinate system --- see \cite{Manoetal, Xuetal}.
It is important to note that in the learning
community one usually looks for {\em stability} of features under the group
action rather than invariance. Note that here we use stability in the CS sense, i.e.\ ``not changed
by much under some operation'', not to be confused with the G.I.T.\ term
of stability which appeared above. Very recently there are
attempts to use
these ideas in conjunction with invariant theoretic methods; see e.g.\
\cite{DM} for an overview.

\subsection*{Acknowledgments}
This work stemmed from an applied project (the reader is encouraged to search
the internet for ``Unity Visual Search by Resonai''), and some of the results
were derived experimentally before being proved.
We thank Omer Sidis, Micha Anholt, Eyal Rozenman, and Lior Davin Lunz for
their help in valuable discussions on the material in this paper, and
implementing pieces of it.
We also thank
Bernd Sturmfels for pointing us to \cite{Ol}, and \cite{GHP}.
%
\section{The main results}\label{S:results}
%
Before delving into the precise set up and statement of the main theorem,
we start by looking at a toy problem, which will motivate both the statement
and the outline of the proof: Suppose we want to construct translation
invariants for smooth curves $C\subset\BR^2$. A functional analytic approach
would be to consider for each curve $C$ the delta function of the curve $\gd_C$,
convolve it with its mirror image, $\gd_C \circ (x\mapsto -x)$, and devise some
way to reconstruct $\gd_C$ -- up to translation -- from this convolution (e.g.\
by analyzing the Fourier coefficients of the result).
A more geometric way would be to utilize the curves normal in order
to build a surface which is ``close enough'' to the second symmetric product of
$C$ with itself.

To this end denote by $N_x C \in \BS^1$ the normal of $C$, and consider the set
$S_C:=\{(x-y, N_xC, N_yC)\mid x, y\in C\}\subset \BR^2\times \BS^1\times \BS^1$. For sufficiently "nice"
curves $C$, and away from a subset of positive co-dimension, $S_C$ is a surface, in a four dimensional space.
Moreover, if
$n\in \BS^1$ satisfies that $\#\{y\in C|N_yC=n\}$ is finite, then the intersection of $S_C$
with $\BR^2\times\BS^1 \times \{n\}$ is simply a finite number of translated copies of $C$,
embedded into a three dimensional space. For a generic value of $n$, those copies are likely not to intersect.

A combined geometric and functional analytic way is then to redefine $\gd_C$ as
the delta function of the image of $C$ in $\BR^2\times \BS^1$, consider the
convolution $(\gd_C \circ(x\mapsto -x)) *_{\BR^2} \gd_C$,
and observe that the support of this distribution is exactly our set $S_C$.
Hence, given the Fourier coefficients of of the convolution, we may recover
$S_C$ with accuracy depending on the number of coefficients.

Let us explain how we will attempt to extend
these ideas to the group $SE(3)$ and to surfaces in $\BR^3$.

First, we have to find the ambient spaces in which we work (the obvious
candidate being, similarly to the toy problem above, $\BR^3\times\BS^2$).
In order to keep the computations tangible, we don't convolve along $SE(3)$,
but rather in two steps:
once along $\BR^3$, and then along $SO(3)$;
essentially using the normality of $\BR^3$ in $SE(3)$. We want the result of
the first convolution to be supported on a 4-fold which is isomorphic to
the second symmetric product of the surface in question with itself,
and the second convolution to be supported on an 8-fold which is isomorphic
to the second symmetric product of that 4-fold. For reasons that will become
clear once we dive into the details of the construction, both isomorphisms are
away from some positive co-dimension subset.

In order to find the relevant ambient spaces we can use in our constriction we
have to perform some rough dimensions estimates: suppose
we have a $k$-dimensional manifold in a $d$-dimensional ambient
spaces on on which an $n$-dimensional group acts (in the example above
$d=3, k=1, n = 2$), then the convolution would give us an $(2 \cdot d - n)$-dimensional
ambient space in which we have a $2\cdot k$-dimensional manifold. If we
repeat this operation twice, for groups of dimensions $n_1, n_2$
then we get a $2\cdot (2 \cdot d -n_1)-n_2$-dimensional ambient space.
Specifically, in the case $n_1 = n_2 = 3$, then we \emph{must} have
$4 \cdot d - 9 \geq 4k = 8\Rightarrow d \geq 4\frac{1}{4}$. Luckily, our
natural candidate $\BR^3\times \BS^2$ is 5-dimensional.
We stress that we do not (yet) claim
that we have showed how to perform the convolutions, and in fact, performing
the second one requires some work; we merely claim that this is what we
\emph{have} to do.

Since we convolve along $SO(3)$, and use the Gauss map to map our
surfaces to $\BS^2$, we will have two constraints: the first is that the Gauss map is generically
finite; the second, is that its
derivative -- the shape operator -- is non-constant.

The first is important since eventually we
want to consider fibers along $\BS^2$, and since the first convolution is along
$\BR^3$. The second requirement is related to non-degeneracy of the convolutions.

Sadly, this means that
simple putative examples like spheres, planes and tori, fail to be good examples
in our
construction\footnote{this is a typical problem in invariant theory; e.g. the
``general genus $g$ curve'' is rather intangible for high $g$}. Nevertheless,
we can use the unit sphere $\BS^2$ as an example to \emph{some} extent, and as long
as we can we can, we will.

Before proceeding we make a couple of remarks on the intuitive view above:
First, it is extremely light from a functional
analytic point of view. Our only usage of functional analysis is to recover
the support of a distribution from its Fourier coefficients. It might be
possible to find some more clever inversion algorithm which will enable denser
representation of the data than an 8-fold to represent a surface.
The second is that similar dimension arguments can be made for manifolds
in $\BR^n$ and the group $SE(n)$.
However, harmonic analysis on these spaces is far more
technical, and the applied use cases are not as clear. Also, one would need
to modify definitions \ref{d:property_star} and \ref{d:property_star_star}
accordingly and the proofs of propositions \ref{p:m2} and \ref{p:m4} in order
to bound the corresponding set of manifolds in which
the inversion algorithm fails. We do not pursue this direction here.

Let us now move to several preliminary definitions, before stating the main
theorem.
\begin{dfn}[Delta function of a surface]\label{D:delta_surface}
  Let $M$ be a compact orientable $C^1$-surface, possibly with a boundary.
  The tempered distribution $\gd_M\in \caD_{cs}(\BR^3\times \BS^2)$ is the unique
  distribution such that for each ball $B\subset\BR^3\times \BS^2$, the value of
  $\int_B\gd_M$ is the surface area of $\{(x, N_xM/\BR^3)\mid x\in M\}\cap B$.
\end{dfn}

We take our test functions from the Sobolev space $W^{2,2} = W^{2, 2}(\mathbb{R}^3 \times \BS^2)$, so $f_n \in W^{2,2}$ converges to $f$ in the topology of $W^{2,2}$ if the functions as well as their first and second derivatives converge to $f$ and its first and second derivatives, respectively.

Note that this identification of surfaces with tempered distributions gives a
natural topology on the space of surfaces via the weak-* topology.
We are using Sobolev space since at some points in the argument it would be important that the topology
respects the differential structure (nearby points have similar first and second
derivatives).

\begin{prd}[the topology on the space of surfaces up to $SE(3)$ action]\label{P:topology}
Recall that the weak-* topology is induced by semi-norms.
Denote the group of rigid orientable transformations on $\BR^3$ by $SE(3)$.
The group $SE(3)$ acts on $\caD_{cs}(\BR^3\times \BS^2)$ via its action
on $\BR^3\times\BS^2$, and the quotient space
$\caD_{cs}(\BR^3\times \BS^2)/SE(3)$ is endowed with the
topology base
\[\{d \in \caD : \inf_{g\in SE(3)} \sum_{i = 0}^n|\langle g_* d,u_i\rangle - a_i| < \epsilon \}\] for arbitrary $\epsilon > 0$, $a_i$ and test functions $u_i$.
\end{prd}

\begin{lma}
The topology on $\caD_{cs}(\BR^3\times \BS^2)/SE(3)$ is Hausdorff.
\end{lma}
\begin{proof}
Fix distributions $d_1$, $d_2$ with compact support such that $g_* d_1 \neq d_2$ for all $g \in SE(3)$. Then, for every $g \in SE(3)$ there is $u$ such that $\langle g_*(d_1) - d_2, u \rangle \neq 0$. For $g$ that contains sufficiently large translation, the supports of $g_* d_1, d_2$ are going to be disjoint, so we can pick any bump function $u$ such that $\langle g_*(d_1) - d_2, u \rangle$  is far from zero at those elements. Now, using continuity and compactness of the subset of $SE(3)$ with bounded translations, we can find finitely many test functions $u_i$ such that for every $g$ there is $i$ such that $|\langle g_*(d_1) - d_2, u_i \rangle| > \epsilon$. Putting them together and letting $a_i = \langle d_2, u_i\rangle$ we obtain an open neighborhood of $d_2$ that does not contain $d_1$. By symmetry, one can also find an open neighborhood of $d_1$ that does not contain $d_2$.
\end{proof}
We remark that this argument is essentially the statement that finitely many sufficiently thin kernels (say, bump functions or restricted Gaussians) are enough in order to distinguish between elements in $\caD_{cs}(\BR^3\times \BS^2)/SE(3)$. In other words, declaring the continuity of the functions \[\min_{g\in SE(3)} \sum \| (g_*(d_1) - d_2) * u_i \|_2\] is enough in order to get a Hausdorff topology on the $SE(3)$-orbits.

\begin{dfn}[Spherical harmonics]
  Denote by $V_d$ the set of homogeneous polynomials of degree $d$ on $\BR^3$.
  Recall the standard definition of {\em spherical harmonics} :
  \[Y_{nm} = (-1)^m\sqrt{\frac{(2n+1)(n-m)!}{4\pi(n+m)!}}P^m_n(\cos\theta)e^{im\phi},\]
  where $n\in\BN, m\in\BZ:|m|\leq n,\phi\in[0,2\pi],\theta\in[0,\pi]$, and
  $P^m_n$ are the {\em associated Legendre polynomials}:
  \[
  P^m_n(x) = \frac{(-1)^m}{2^n n!}(1-x^2)^{m/2}\frac{d^{n+m}}{dx^{n+m}}(x^2-1)^n.
  \]
  Recall that the spherical harmonics are orthonormal basis for the Hilbert space
  $L^2(\BS^2)$ up to the scaling factor $4\pi/(2n+1)$, (we identify a point in $\BS^2$ with its polar coordinates).
  Let
  $[n]:=\langle Y_{n,m}:|m|\leq n\rangle$.
  In order to work with irreducible representations over the real numbers,
  we define
  \begin{equation}\label{e:real_ynm}
  Y_{nm}^\BR:=
  \begin{cases*}
  \sqrt2\text{Im} Y_{n|m|}& if $m < 0$\\
  Y_{00}& if $m=0$ \\
  \sqrt2\text{Re}Y_{n|m|}& if $m > 0$
  \end{cases*}.
  \end{equation}
\end{dfn}
\begin{prp}
  Taking $x,y,z$ as a basis of $\BR^3$, and
  making the standard identifications
  \[
  r^2=x^2+y^2+z^2,\quad\cos\theta=z/r,\quad\exp i\phi= (x+iy)/\sqrt{x^2+y^2},
  \]
  we have $V_0 = [0]$, $V_1=r[1]$, and for all $n\geq 0$: $V_{n+2}=r^2V_n\oplus r^{n+2}[n+2]$.
\end{prp}
\begin{dfn}\label{d:rho}
  Define the map
  \[\begin{aligned}
  \rho_{n,m}:\caD_{cs}(\BR^3\times \BS^2)&\to\caD_{cs}(\BR^3)\\
  \mu\mapsto&\int_{\BS^2} Y_{nm}^{\BR}d\mu.
  \end{aligned}\]
  We think of the vector space $\langle \rho_{n,m}(\mu)\mid n,m \rangle$
  as residing in
  \[\caD_{cs}(\BR^3)\otimes L^2(\BS^2)^\vee\cong\caD_{cs}(\BR^3)\otimes(\oplus_{n} [n]^\vee).\]
\end{dfn}
\begin{exm}[The unit sphere running example]
  Let us compute the moments of $\rho_{n,m}(\gd_{\BS^2})$.
  The moments reside in the dual space of the polynomials
  over $\BR^3$, which we represent here as
  $\oplus_{n\leq d} V_n$ with the basis $r^aY^{\BR}_{bc}$.
  Computing we get:
  \[
  \int_{\BR^3\times\BS^2}r^aY_{bc}Y_{nm}\gd_{\BS^2}=
  \int_{\BS^2}Y_{bc}Y_{nm}=(-1)^m\int_{\BS^2}Y_{bc}Y_{n,-m}^*
  \]
  This integral is non trivial precisely when $b=n, c=-m$, in which case
  it is equal to $(-1)^m4\pi/(2n+1)$. Using Equation \ref{e:real_ynm} we
  deduce the coefficients in terms of $Y_{nm}^\BR$.
\end{exm}
\begin{dfn}\label{d:fF}
  Denote by
  $f^{\mu}_{d, n,m,n',m'}\in \oplus_{j\leq d} V_j$
  the truncation at degree $d$ of the moment generating function of
  \[(\rho_{n',m'}(\mu)\circ (x\mapsto-x))*_{\BR^3}\rho_{n,m}(\mu),\]
  where the exponential function is considered
  as a sum of powers.
  The sequence $f^\mu_{d,\bullet} = \langle f^{\mu}_{d, n,m,n',m'} \mid n,m, n',m' < d'\rangle$ represents
  an element in  $(\oplus_{n\leq d} V_n)\otimes(\oplus_{n\leq d'} [n]^\vee)^{\otimes^2}$.
  Denote by $F^{\mu}_{d, d'}$ the image in the trivial $SO(3)$ subspace
  of the convolution of $f^\mu_{d,\bullet}$ with a copy of itself composed
  with the inverse map on $SO(3)$. This can be viewed as the tensor product of
  $f^\mu_{d,\bullet}$ with itself.
  Fixing a basis for
  $(\oplus_{n\leq d} V_n)\otimes(\oplus_{n\leq d} [n]^\vee)^{\otimes^2}$,
  we can think of its second tensor -- in which $F^{\mu}_{d, d'}$ lives --
  as a finite dimensional coefficients spaces. Note that the $F^{\mu}_{d, d'}$
  are degree $4$ polynomials in the moments of the functions $\rho_{n,m}(\mu)$.
\end{dfn}
\begin{rmr}[computing the coefficients of $f$ from the moments of the
    $\rho_{n,m}$s]
  The moments of $\rho_{n,m}(\gd_M)$ are expressed in
  terms of the basis $r^aY^{\BR}_{bc}$, and we want to compute the moments of
  $f^{\gd_M}_{d,n,m,n',m'}$ -- for fixed $n,m,n',m'$
  in terms of the same basis.
  Computing these moments is straight forward
  (though computationally
  heavy) using two observations: The
  first is that it is immediate to move between monomial basis to homogeneous
  polynomials of degree $a$, and the $r^aY^{\BR}_{bc}$s. The second is
  that the moment generating function of $f$ is the product of the
  moment generating functions of $\rho_{n,m}$ and $\rho_{n',m'}$, and that
  in monomial basis the coefficients of the moment are immediate to compute.
\end{rmr}
Our main theorem is:
\begin{MainThm}
  The map $M\mapsto\gd_M\mapsto \bigcup_{d,d'\in\BN} F_{\gd_M,d, d'}$
  is injective on
  a comeagre subset of compact $C^3$-surfaces in the weak-* topology.

Moreover, we give an effective
reconstruction algorithm which, given
$F_{\gd_M,d, d'}$ where $d,d'$ are bounded by some $K$,
and assuming $M$ as above with a known bound from above
on the absolute value of the sectional curvature, and on the radius,
recover $M$ up to some accuracy which depends on $K$, the bounding radius, and
the bound on the curvature.
This algorithm is continuous in the
coefficients $F_{\gd_M,d, d'}$,
and the complexity of the algorithm is polynomial at $K$.
\end{MainThm}
The rest of the paper proceeds as follows:
As already mentioned in the introduction, the gist of our
construction is convolving $\gd_M$
with its reflection, first along $\BR^3$, and then along $SO(3)$. In Sections
\ref{s:translations}-\ref{S:rotation_inv} we give a geometric interpretation of
this statement, and
show that under certain conditions this ``double convolution'' is an
invertible operation, up to the action of $SE(3)$. In
Section \ref{S:rep} we show how to interpret the construction as a convolution
with the reflection, and recap some classical materiel --- interpreting the
convolution with the reflection in representation theoretical terms. Finally
in Section \ref{S:functional_analysis} we show that any sufficiently general
surface
satisfies our conditions, and conclude by using the representation theoretical
aspects developed in the previous section to complete the effective part of
the proof.

We remark that the invariants make sense for every $C^1$-surface, and the conditions $(\star)$ and $(\star\star)$ which imply the invertibility rely only on the second derivative, see definitions \ref{d:property_star} and \ref{d:property_star_star}. We are using the assumption that the surface is in $C^3$ only in Lemma \ref{c:co_meagre} in order to verify that the subset of bad points which is of measure zero would also be of low $\mathcal{I}$-dimension (see definition \ref{d:I-dimension}).
%
\section{Geometric aspects I: preliminaries and translation invariance}\label{s:translations}
%
We start this section by considering again our toy model from the beginning
of the former section: A curve in $\BR^2$, which we ``lift'' to
$\BR^2\times\BS^1$. There we claimed (without proof) that
$S_C\subset\BR^2\times\BS^1\times\BS^1$ is ``close'' in some intuitive sense to
the second symmetric product, and that its fibers along general points in
$\BS^1$ are composed of copies of $C$. We maintain that this is easy for the
curve case. However, it is harder to do when the dimensions at hand increase,
which happens in this section and even more so in section \ref{S:rotation_inv}.
Our solution borrows ideas from intersection theory. Before presenting it
in the generality we need we will present it in the toy model context.
Define the following sets in $\BR^2\times\BR^2\times\BS^1$:
\[
Y:=\{(g,x, n_x)|g\in\BR^2,x\in c\}, \quad Z:=\{(g,g+x, n_x)|g\in\BR^2,x\in c\}.
\]
Then $X:=Y\cap Z$ parameterizes the points in $S_C\cap(\BR^3\times\{(n,n)|n\in\BS^1\})$.
If $X$ has a component of dimension $>1$ then the Gauss map behaves badly. For example,
it has a non discrete fiber.
In higher dimensions the analog of $X$ (and it projections
on the components of $Y, Z$, and subspaces of $X$ we may define)
exposes much more information on the degeneracy. That being said, we use $X$
as a tool to {\em bound} the degeneracy, not describe it exactly.

We note that unlike in the toy problem, in our actual result we will have
to deal both with the fact that not all but merely generic fibers (the analogue
of the fibers we took on $\BS^1$) work, and even on them we will
have to deal with copies of our surfaces intersecting one another (but luckily,
only isolated points).
\begin{ntt}
  Throughout the rest of the paper we use capital letters to denote varieties,
  and small letters to denote points.
  We use analogue constructions to deal with actions of the groups $\BR^3$ and
  $SO(3)$; however, as the varieties will have different dimensions for
  these two groups, we will use the same letter, but with a different
  subscript index, which will be the \emph{dimension} of the variety involved.
  The mild abuse of this rule is varieties which are symmetric
  powers of our surface of study $M=M_2$ - up to some positive co-dimension
  subset:
  For $M=M_2$ itself we drop the subscript altogether,
  and the other varieties are
  denote by $M_{2\times n}^\bullet$, where $n$ is the power, and $\bullet$ is
  some extra annotation, depending on the discarded co-dimension $\geq2$
  subset. We would indicate explicitly the other rare cases in which
  we do not follow this rule. The reason for this choice of notation is that
  the main argument in this section repeats, with some crucial
  differences, in Section \ref{S:rotation_inv}. We want the notation to
  reflect this similarity, and yet keep it so that we can refer the the objects
  appearing in this section in the following sections.
  Finally, the only normal bundle considered throughout this work is the normal
  to $M\subset \BR^3$. Hence we denote $N_x:=N_xM/\BR^3$.
\end{ntt}
\begin{dfn}\label{d:I-dimension}
  Let $W$ be a manifold. Denote by $\mathcal{I}_n(W)$ the collection of
  subsets of $W$ which are covered by a countable union of
  {\em smooth images} of the
  closed $n$-dimensional Euclidean ball.
When $W$ is clear from the context, we would say that $X$ has $\mathcal{I}$-dimension $\leq n$ if $X\in\mathcal{I}_n(W)$.
\end{dfn}
By the invariance of domain, for  $W$ of dimension above $n$, $\mathcal{I}_n(W)$ is a proper, $\sigma$-complete ideal, and every member of it is meagre and null (of measure zero).
Clearly, every image of a set from $\mathcal{I}_n(W)$ under a smooth map from $W$ to $W'$ is
in $\mathcal{I}_n(W')$.

\begin{ntt}
  We consider the action of $G=\BR^3$ on $\BR^3$ by translation. We use
  different notations for these two copies of $\BR^3$ since from our perspective
  the first is a group whereas the latter is a homogeneous
  space of this group.
  Let $M\subset \BR^3$ be a compact twice differentiable surface.
  Denote the inclusion map $M\to\BR^3$ by $i_{M/\BR^3}$.
  We typically denote the action of $g \in G$ on $m \in \BR^3$ by $g\circ m$.
  In cases in which this notation might cause confusion, we would denote
  the action of $g\in G$ on $\BR^3$ by $\eta_{g,\BR^3}$.

  We define two 5-folds and one surface in $G\times\BR^3\times \BS^2$:
  \[
  \begin{aligned}
    Y_5=Y_5(M):=&\{(g, m, N_m)\mid g\in G, m\in M\},\\
    Z_5=Z_5(M):=&\{(g, g\circ m, N_m)\mid g\in G, m\in M\},\\
    X_2=X_2(M):=&Y_5\cap Z_5.
  \end{aligned}
  \]
  Note that $Y_5$ is naturally isomorphic to $G\times M$, and so there is a
  natural bundles isomorphism: $TY_5=TG\oplus TM$.
  Moreover, there are pointwise isomorphisms
  $T_{(g, g\circ m, N_m)}Z_5\cong T_gG\oplus T_{m}M$.
\end{ntt}
\begin{exm}[The unit sphere running example]
  If $M=\BS^2$, then
  \[
  \begin{aligned}
  Y_5&=\{(g, n, n)\mid g\in \BR^3, n\in\BS^2\},\\
  Z_5&=\{(g, g + n, n)\mid g\in \BR^3, n\in\BS^2\},\\
  X_2&=\{(0, n, n)\mid n\in\BS^2\}.
  \end{aligned}
  \]
  Similarly, it is easy to verify that whenever the Gauss map is injective
  on $M$, then $X_2 = \{(0, x, N_x)\mid x\in M\}$).
\end{exm}

\begin{exm}[Unit disc through the origin non-example]
  Let $M$ is the unit disc $D_{1,0}\subset\BR^2\subset\BR^3$, where
  $\BR^2\subset\BR^3$ are the points whose dot product with the north pole
  $N\in\BS^2$ is trivial. Then we have:
  \[
  \begin{aligned}
  Y_5&=\{(g, x, N)\mid g\in\BR^3, x\in D_{1,0}\},
  Z_5&=\{(g, x+g, N)\mid g\in\BR^3, x\in D_{1,0}\},
  \end{aligned}
  \]
  and so \[X_2=\{(g,x, N)\mid x, x+g\in D_{1,0}\}\] is 4-dimensional.
\end{exm}
\begin{exm}[The cylinder non-example]
  If $M=\BS^1\times[-1,1]$ then
  \[\begin{aligned}
  Y_5&=\{((a,b), (x, y), x)\mid a\in\BR^2, b\in\BR,x\in\BS^1, y\in[-1,1]\},\\
  Z_5&=\{((a,b), (a+x, y+b), x)\mid a\in\BR^2, b\in\BR,x\in\BS^1, y\in[-1,1]\},
  \end{aligned}\]
  and so \[X_2=\{((0,b), (x, y), x)\mid y,y+b\in[-1, 1], x\in\BS^1\}\]
  which is 3-dimensional.
\end{exm}
\begin{exm}[A union of two spheres example]
  If $M$ is a union of $\BS^2$ , and the translate of $\BS^2$
  by $x$ (where $\|x\|>2$), then
  (although $M$ is disconnected):
  \[
  \begin{aligned}
  Y_5&=\{(g, n, n), (g, n+x, n)\mid g\in \BR^3, n\in\BS^2\},\\
  Z_5&=\{(g, n + g, n), (g, n + g+ x, n)\mid g\in \BR^3, n\in\BS^2\},
  \end{aligned}
  \]
  and so
  \[X_2=\{(0, n, n), (x, n+x, n), (-x, n-x, n), (0, n + x, n)\mid n\in\BS^2\},\]
  which is isomorphic to a union of four spheres.
\end{exm}
\begin{ntt}
Let us return to the general case.
  Denote the inclusion $X_2\to G\times\BR^3\times \BS^2$ by
  $i_{X_2/G\times\BR^3\times \BS^2}$.
  Denote the composition of the inclusion map and the projection maps from
  $X_2$ to the first, second
  and third components in $G\times\BR^3\times \BS^2$ by
  \[
  \pi_{X_2/G}, \quad \pi_{X_2/\BR^3}, \quad \pi_{X_2/\BS^2}
  \]
  respectively. Since $\range(\pi_{X_2/\BR^3})\subset M$,
  we may also define $\pi_{X_2/M}:=i^{-1}_{M/\BR^3}\circ\pi_{X_2/\BR^3}$.

  Finally we denote
  \[
  \rho_{X_2/M}:=\eta_{\pi_{X_2/G}(-)^{-1},\BR^3}\circ\pi_{X_2/M}=
  i^{-1}_{M/\BR^3}\circ\eta_{\pi_{X_2/G}(-)^{-1},\BR^3}\circ\pi_{X_2/\BR^3},
  \]
  where the composition in the right term of the last equality is well defined by
  the definition of $X_2$. Namely, $\rho_{X_2/M}(g, x, n) = g^{-1} \circ x \in M$.
\end{ntt}
Each $x \in X_2$ corresponds to a point in $Y_5$ and a point in $Z_5$; each
of these points has three components: a $G$-component, an $M$-component,
and an $\BS^2$-component.
The function $\pi_{X_2/M}$ returns the $M$-component of the
point corresponding to $Y_5$ and $\rho_{X_2/M}$ returns $M$-component of
the point corresponding to $Z_5$. Rewriting the definition of $X_2$, we get for all $p \in X_2$,
\begin{equation}\label{equation:rho and pi translation}
\pi_{X_2/M}(p) = \eta_{\pi_{X_2/G}(p),\BR^3}(\rho_{X_2/M}(p)).
\end{equation}
As the notation suggests, we expect $X_2$ to be a two dimensional manifold.
Since $\{0\} \times M \subseteq X_2$, $X_2$ always contains a subset of dimension 2. The following definition isolates a rather large class of surfaces which behave as desired, at most places.
\begin{dfn}[Property $\star$]\label{d:property_star}
  We say that $M$ has property $\star$ if there is a set $X_1' \subseteq X_2$
  of $\mathcal{I}$-dimension $\leq 1$ such that for any
  $p=(g, m, N_m)=(g, g\circ m', N_{m'})\in X_2\sm X'_1$
  the following properties hold:
  \begin{enumerate}
    \item\label{i:iso_star}
      $\rank S_m = \rank S_{m'} = 2$.
    \item\label{i:shape_star}
      Denote $n:=N_m = N_{m'}$.
      Then the linear map obtained from the composition of the maps
     \[
     T_m M\xrightarrow{(id,g_*)} T_m M\oplus T_{m'} M\xrightarrow{d N_m,d N_{m'}}
     T_n \BS^2\oplus T_n \BS^2
     \xrightarrow{(x,y)\mapsto x-y} T_n \BS^2,
     \]
     is of rank $2$.
  \end{enumerate}
\end{dfn}
Before proceeding to tackle translation invariance, recall the
definition of the {\em shape operator} and prove two
properties, related to
$\star$. The first evaluates the components of the tangent space of $X_2$ at a
point $p$, and the other is a local consequence of requirement
\ref{i:iso_star} of $\star$, that
enables us to compute the dimension of $X_2$. There are a couple of
equivalent definitions for the shape operator. We are using the following one:
\begin{dfn}[The shape operator]
  Let $A\subset\BR^3$ be a $C^2$ surface, $a\in A$ a point
  then the \emph{shape operator} is the linear map $S_a\colon T_aA\to T_aA$ which takes
  $v$ to $-d N_aA$.
\end{dfn}
\begin{lma}\label{L:prop_star1_is_regualrity}
Let $M\subseteq \BR^3$ be a twice differentiable compact surface,
let $x \in X_2$ and let $m = \pi_{X_2/M}(x), m' = \rho_{X_2/M}(x)$ be the
corresponding manifold points.
Then if requirement \ref{i:iso_star} of condition $\star$ holds for $x$ then
$(\pi_{X_2/M})_*$ and $(\rho_{X_2/M})_*$ are injective.
In particular, around $x$, $\dim X_2 \leq 2$.
\end{lma}
\begin{proof}
  Let $v\in\ker{\pi_{X_2/M}}_*|_{T_xX_2}$.
  So, the image of $v$ under the map
  \[
  (\pi_{X_2/G},\pi_{X_2/M},\pi_{X_2/\BS^2})_*
  \colon  T_x X_2\to T_{g} G \oplus T_{m} M \oplus T_{s} \BS^2
  \]
  is of the form $(h, 0, t)$.

  Let $w:={\rho_{X_2/M}}_*(v)$. Then, by the chain rule applied to
  Equation \ref{equation:rho and pi translation}, $w$ is the
  derivative of $\eta_{g^{-1},\BR^3}$, at $g=h$, operating on $0$,
  which is simply $0-h$. Hence:
  \[
  0 = S_{m}(0) = t = -S_{m'}(w) = S_{m'}(h)
  \]
  i.e., $h \in \ker S_{m'}$.

  The proof of the second part is similar.
\end{proof}

\begin{lma}\label{lemma:rank piX/G translations}
  A point $p\in X_2$ satisfies requirement \ref{i:shape_star} if and only if \[\rank (\pi_{X_2/G})_* = 2.\]
\end{lma}
\begin{proof}
  Let $p = (g, m, N_m) = (g, g\circ m', N_{m'})$ (where clearly $g=m-m'$).
  So the image of $p$ under the map $(\pi_{X_2/M},\rho_{X_2/M})$ is $(m, m')$.
  Our claim is that the map induced from the composition of maps
  \[
  X\xrightarrow{p\mapsto (m, m')} M\times M\hookrightarrow \BR^3\times \BR^3\xrightarrow{(x,y)\mapsto x-y} G
  \]
  on the corresponding tangent spaces is of rank 2.
  To this end we have to show that in the compositions of linear maps below
  \[
  T_p X_2\to T_mM\oplus T_{m'}M\hookrightarrow T_m\BR^3\oplus T_{m'}\BR^3\to T_{m'-m}G,
  \]
  the image of the left map intersects trivially the kernel of
  the right map (where all the maps are the bundle
  maps induced from the maps above).

  A vector $(x, y) \in T_m \BR^3 \times T_{m'} \BR^3$, belongs to the kernel
  of the left map, if and only if $(\eta_{-m,\BR^3})_*(x) - {\eta_{m',\BR^3}}_*(y) = 0$ if and only if
  $y = g_* x$.
On the other hand, let $t \in T_{p} X_2$. Let us denote
  $x = (\pi_{X_2/M})_*(t)$, $y = (\rho_{X_2/M})_*(t)$.
  Since
  $n = N_{\pi_{X_2/M}(\ti{p})}=N_{\rho_{X_2/M}(\ti{p})}$,
  by the definition of $X_2$,
  \[
  dN_m (x) = dN_{m'} (y).
  \]
  Therefore, by requirement \ref{i:shape_star} in $(\star)$,
  $x$ and $y$ have to be trivial.
\end{proof}
\begin{lma}\label{L:gauss_map_finite}
Let $p \in \BS^2$, and assume that $p$ is not in the image of the projection of
$X_1'$ onto $\BS^2$. Then, there are only finitely many $m \in M$ such that $N_m = p$.
\end{lma}
\begin{proof}
  Let $F = \{m \in M \mid N_m = p\}$. By the continuity of the Gauss map,
  $F$ is a closed set. Moreover, if
$F$ is infinite, then by the compactness of $M$, $F$ has an accumulation point.

But such a point would necessarily be in $X_1'$ (as otherwise, by requirement \ref{i:iso_star}, the Gauss map would be locally injective around this point), and thus $p$ belongs to the projection of $X_1'$.
\end{proof}
\begin{ntt}\label{n:m2}
  Let $M\subset\BR^3$ be a twice differentiable compact surface.
  Denote by $M_{2\times2}'\subset \BR^3\times (\BR^3\times \BS^2)^2$
  the set
  \[
  \{(t, (m_1, N_{m_1}), (m_2, N_{m_2}))|m_1,m_2\in M, m_1-m_2=t\}.
  \]
  Denote by $M_{2\times 2}\subset\BR^3\times (\BS^2)^2$ the projection of
  $M_{2\times 2}'$ on the
  product of the first $\BR^3$ and the two copies of $\BS^2$; denote the
  projection $M_{2\times2}'\to M_{2\times2}$ by $\pi_{M_{\times2}'/M_{2\times2}}$.
  Finally denote by
  $\gD_{\BS^2}$ the diagonal of $(\BS^2)^2$.
\end{ntt}
\begin{exm}[The unit sphere running example]
  If $M=\BS^2$ then $M_{2\times2}=\{(n-n',n,n')\mid n,n'\in\BS^2\}$.
\end{exm}
\begin{exm}[The torus non-example]
Let us take $M$ to be the torus, with the parametrization
\[M = \{((R + r \cos(\theta))\cdot \cos (\varphi), (R + r \cos(\theta))\cdot \sin (\varphi), r \sin(\theta) ) \mid \theta, \varphi \in [0, 2\pi)\}.\]
Computing the normals and the shape operator one can see that for $\theta$ such that $\cos \theta \neq 0$ (i.e.\ not the top or bottom circle), the Gauss map is 2-to-1 and the shape operator is of full rank. In particular, $X_2$ is two dimensional. For every $a, b \in M$ such that $N_a = N_b$ and they are not in the top or bottom circle,
\[a - b = (2 R \cos \varphi, 2 R \sin \varphi, 0)\]
for some $\varphi \in [0, 2 \pi)$. This implies that $M$ fails to satisfy $(\star)$ (as the projection to $G$ is of rank $1$).

Indeed, proposition \ref{p:m2} ahead fails for the torus. For any fiber, expect a small set, we will obtain a pair of tori intersecting on a circle.
\end{exm}

Let $g \in \BR^3$. Let $\ti{M} = g + M$, the action of $g$ on $M$. It is clear that applying the definition of $M_{2 \times 2}$ for $g+ M$ would result in the same object, since the normals remain unmodified by the translation and so are the differences between pairs of points. Thus we get:
\begin{prp}\label{p:trans_inv}
  the construction of $M_{2\times 2}$ is invariant under the actions of
  translations on $\BR^3$.
\end{prp}
We are now ready to state the main result of this section.
\begin{prp}\label{p:m2}
  Suppose $M$ satisfies property $\star$, then there is an open dense
  subset $W\subset \BS^2$, of full measure (relative to the image of the Gauss map),
  so that for any point $p\in W$, the intersection
  of $M_{2\times2}$ with $\BR^3\times\{p\}\times \BS^2$ is non-empty and
  composed of a finite number
  of translations of $M$, $\{g_1\circ M, \dots, g_n\circ M\}$, intersecting at most
  on isolated points along the pullback of the diagonal and its antipodal image, $\gD_{\BS^2}\subset(\BS^2)^2$,
  under the projection from $\BR^3\times(\BS^2)^2$ on the second component.

  Moreover, each of these copies is continuous in $p$, in the
  sense that the set of $g_i\in G$ is continuous in $p$.
\end{prp}
\begin{proof}
  First, the set $W$ of all $p$ satisfying the conclusion of the
  proposition is clearly open. We will show that its complement
  has $\mathcal{I}$-dimension $\leq 1$. This is enough, as the image of the Gauss
  map on $\BS^2$ has positive measure.

  Let $m_1,\ldots,m_4\in M$ be four points so that
  \[
  \begin{matrix}
  (m_1-m_2, (m_1, N_{m_1}), (m_2, N_{m_2})),&\\
    (m_3-m_4, (m_3, N_{m_3}), (m_4, N_{m_4}))&\in M_{2\times2}'
  \end{matrix}
  \]
  are both mapped to the same point in $M_{2\times2}$; i.e.\ we have:
  \begin{equation}\label{e:orig_r3_cond}
  m_1-m_2=m_3-m_4,\quad N_{m_1}=N_{m_3},\quad N_{m_2}=N_{m_4}.
  \end{equation}
  Reordering the equations, we have that
  \[x_i = (m_i - m_{i + 2}, m_i, N_{m_i}) \in X_2, \text{ for } i = 1,2.\]

  Note that the fiber $p$ is the projection to $\BS^2$ of the point $x_2$.

  Let us bound from above the solution set to this equation, partitioning to
  three cases, where at each case $c$ we obtain a closed set $F_c$,
  of $\mathcal{I}$-dimension $\leq 1$ such that
  for every $p \notin F_c$, the corresponding case behaves nicely:

  {\em Case 1:} The diagonal case; i.e.\ $N_{m_1}=\pm N_{m_2}$ (and then
  $N_{m_3}=\pm N_{m_4}$). Let
  \[F_1 = \pi_{X_2/\BS^2}(X_1') \cup a \circ \pi_{X_2/\BS^2}(X_1'),\]
  where $a(x) = -x$ is the antipodal map.
   This is a set of $\mathcal{I}$-dimension $\leq 1$.

  For every $p \notin F_1$, by lemma \ref{L:gauss_map_finite}, the fiber
  of the Gauss map is finite.
  Note that we know the location of the intersection points beforehand,
  since they are on the pullbacks of the diagonals.

  {\em Case 2:}
  Assume $p \notin F_1$. So, we have (by the definition) that $x_2 \notin X_1'$.
  Further assume that $x_1\in X'_1$.
  We rearrange the left equality of Equation \ref{e:orig_r3_cond}, to the
  form: $m_1 - m_3 = m_2 - m_4 \in \mathbb{R}^3$. Since on
  $X_2 \sm X_1'$ the projection to $\mathbb{R}^3$ is of rank $2$ by lemma
  \ref{lemma:rank piX/G translations} we conclude that the
  sub-manifold of solutions to
  this equation on $X_2 \times X'_1$ is at most of one $\mathcal{I}$-dimensional,
  and more precisely for
  each point $x' \in X'_1$, the set of $x \in X_2 \sm X'_1$
  such that $(x', x)$ corresponds to a
  solution of Equation \ref{e:orig_r3_cond} consists
  of isolated points, and thus at most countable.
  Thus, this case contributes at most one $\mathcal{I}$-dimensional set to the solutions
  of Equation \ref{e:orig_r3_cond}. Let $F_2$ be the projection of this
  set to $\BS^2$.

  {\em Case 3:}
  $x_1, x_2 \in X_2\sm X'_1$,
  and the corresponding normals are different (since we have already dealt
  with the diagonal case).
  There is a
  chart $U\subset\BR^2$ and a chart-map $f\colon U\to M$, so that $m_i\in f(U)$.
  Let $p_1,\ldots,p_4\in U$ be four points so that $f(p_i)=m_i$.
  The points in $M_{2\times2}$ for which
  $M_{2\times2}'\to M_{2\times2}$ is not one to one -- and
  where all four relevant points of $M$ are in $U$ -- are
  in one to one correspondence
  with the $f$ images of the solutions set of
  {\small
  \begin{equation}\label{e:r3cond}
    f(q_1)-f(q_2)=f(p_3)-f(q_4), \quad \frac{\nabla f|_{q_1}}{\|\nabla f|_{q_1}\|}
    =\frac{\nabla f|_{q_3}}{\|\nabla f|_{q_3}\|},\quad
    \frac{\nabla f|_{q_2}}{\|\nabla f|_{q_2}\|}
    =\frac{\nabla f|_{q_4}}{\|\nabla f|_{q_4}\|},
  \end{equation}
  }
  for $(q_1, q_2, q_3, q_4)\in U^4$.
  Denote by $A\subset M^4$ the solution set of Equation \ref{e:r3cond},
  corresponding to the {\em current case},
  and let $V$ be the pre-image of $A$ in $U^4$.
  Define the map:
  \[
  \begin{aligned}
    \ti{f}:V&\to \BR^3\times \BR^3\times \BR^3\\
    (q_1,q_2,q_3,q_4)&\mapsto  \big(f(q_1)-f(q_2) - f(q_3) + f(q_4),\\
    \ &\qquad\qquad N_{f(q_1)} - N_{f(q_3)}, N_{f(q_2)} - N_{f(q_4)} \big).
  \end{aligned}
  \]
  Note that by lemma \ref{L:prop_star1_is_regualrity} of $\star$ we may assume
  without loss of generality (via local change of
  coordinates), that $p_i=(x_i, y_i)$ are such that:
  \[
  \frac{\partial f}{\partial x_1, y_1}|_{p_1}=\frac{\partial f}{\partial x_3, y_3}|_{p_3},
  \quad
  \frac{\partial f}{\partial x_2, y_2}|_{p_2}=\frac{\partial f}{\partial x_4, y_4}|_{p_4}
  \]
  Let us compute the Jacobian matrix of this map about $p_1,\ldots,p_4$
  with respect to local coordinates
  $x_i, y_i$ about the $p_i$-s.
  This is the 9 rows by 8 columns matrix:
  \[J\ti{f}=\text{rows reorder}\begin{pmatrix}
  0_{3\times 2} &\frac{\partial N_{m_2}}{\partial x_2, y_2}&
  0_{3\times 2} &-\frac{\partial N_{m_4}}{\partial x_4, y_4}\\
  \frac{\partial N_{m_1}}{\partial x_1, y_1}&0_{3\times 2}&
  -\frac{\partial N_{m_3}}{\partial x_3, y_2}&0_{3\times 2}\\
  \frac{\partial f}{\partial x_1, y_1}|_{p_1}&
  -\frac{\partial f}{\partial x_2, y_2}|_{p_2}&
  -\frac{\partial f}{\partial x_3, y_3}|_{p_3}&
  \frac{\partial f}{\partial x_4, y_4}|_{p_4}
  \end{pmatrix},\]
  where each entry stands for a 3 rows by 2 columns sub matrix.
  Let us concentrate on the two left-most columns (i.e.\ the left
  ``block column'') above: The bottom rows are given by
  $(\frac{\partial f}{\partial x_1}|_{p_1},\frac{\partial f}{\partial y_1}|_{p_1})$,
  which span $T_{m_1}M$. As for the middle columns, they are, by definition,
  $-S_{m_1}(\frac{\partial f}{\partial x_1}|_{p_1}),
  -S_{m_1}(\frac{\partial f}{\partial y_1}|_{p_1})$,
  where $S$ is the shape operator (i.e.\ $-dN$).
  Adding columns 1-4 to columns 5-8 and eliminating the (trivial) first and
  forth rows, columns 5-8 become:
  \begin{equation}\label{e:submatrix_translation}\begin{pmatrix}
  0_{2\times 2}&(S_{m_1}-S_{m_3})(\frac{\partial f}{\partial x_1, y_1}|_{p_1})\\
  (S_{m_2}-S_{m_4})(\frac{\partial f}{\partial x_2,y_2}|_{p_2})&0_{2\times 2}\\
  0_{3\times 2}&0_{3\times 2}
  \end{pmatrix},\end{equation}
  where again, each of the blocks in the matrix is a 3 rows by 2 columns matrix.
  We now observe that the pair $f(p_1), f(p_2)$ (respectively $f(p_3), f(p_4)$)
  corresponds to a point on $X_2$; i.e.\ this pair sits on some surface in $U^2$.
  Furthermore,
  by condition \ref{i:shape_star} of property $\star$, the top three rows
  (respectively the
  middle three rows) of this matrix have rank $2$ on a neighborhood of
  $(p_1, p_3)$ (respectively $(p_2, p_4)$) out of some curve in $U^2$.
  Finally, the lower left 3 rows by 4 columns block of $J\ti{f}$ is of rank
  $3$ out of the pullback under the Gauss map (on both copies of $\BS^2$) of
  $\gD_{\BS^2}$.

  By the analysis above, the rank of the Jacobian is 7,
  and therefore the solution space is a collection of curves in $U^4$.
  Taking a suitable countable cover of $M$, outside the sets from Cases 1 and 2,
  we conclude that this case consists of at most countable many
  one dimensional curves in $M^4$, and thus it is a set
  of $\mathcal{I}$-dimension $\leq 1$ (i.e.\ the set is contained in a
  countable union of smooth images of the segment $[-1, 1]$).

  Let $F_3$ be the projection of the solution set of
  Equation \ref{e:r3cond}, under this case, to $\BS^2$, via $N_{m_1}$.

  Fix $n \in \BS^2$ in the open set away from $\bigcup F_c$. Then, automatically,
  any solution for Equation \ref{e:r3cond}, $(p_1, p_2, p_3, p_4)$ with
  $N_{f(p_1)}= N_{f(p_3)} = n$ has to be in the pullback of $\Delta_{\BS^2}$,
  as wanted.

  The ``moreover'' part of the claim follows straightforward from continuity.
\end{proof}

\begin{rmr}\label{remark:dimension-of-m2x2}
Let $M$ be a compact smooth manifold. Then $M_{2 \times 2}$ is composed of a compact manifold of dimension $\leq 4$, together with a compact set which is a projection of a manifold of positive co-dimension in $M\times M$.

If $M$ satisfies $\star$, then $M_{2 \times 2}$ contains a four dimensional manifold.
\end{rmr}
\begin{proof}
Consider the map $M \times M \to M_{2 \times 2}$ given by
\[(m,m')\mapsto (m - m', N_m, N_{m'}).\]

A vector $(t, t') \in T_m M \times T_{m'} M$ belongs to the kernel of the map
induced on the tangent spaces by the above map, which happens if and only if:
\[t - t' = 0, \quad S_m(t) = 0, \quad S_{m'}(t') = 0.\]
Thus, we can compute the rank of this map, by splitting into cases:
\begin{enumerate}
\item The map is an immersion if:
\begin{enumerate}
\item $\rank S_m = 2$ or $\rank S_{m'} = 2$, or
\item $\rank S_m = \rank S_{m'} = 1$ and $\ker S_m \cap \ker g_*S_{m'} = \{0\}$, for $g = m' - m$.
\end{enumerate}
\item Otherwise, it is of rank 3 if:
\begin{enumerate}
\item there is a non-zero $t \in T_m M$ such that $S_m(t) = g_*S_{m'}(t) = 0$, and $\rank S_m = 1$ or $\rank S_{m'} = 1$ ($g = m' - m$), or
\item $S_m = 0$ and $S_{m'} = 0$ but $N_m \neq N_{m'}$.
\end{enumerate}
\item It is of rank 2 if $S_m = 0$ and $S_{m'} = 0$ and $N_{m} = N_{m'}$.
\end{enumerate}

The points corresponding to Case 1 would contribute a four dimensional manifold (besides the self intersections). Similarly, the points corresponding to Case 2 would contribute at most a three dimensional manifold, and the points corresponding to Case 3 would contribute a two dimensional manifold.

In particular, assuming $\star$, by requirement \ref{i:iso_star}, in order to avoid the first case, both $m$ and $m'$ are in $\pi_{X_2/M}(X_1')$, which is (at most) one dimensional by $\star$. So, outside a closed set of co-dimension $2$, this map is an immersion. Moreover, by proposition \ref{p:m2}, outside this closed set, the self intersections are given by a finite union of two-dimensional sets (Case 1) and one dimensional sets (Case 2 and Case 3).

Thus $M_{2\times 2}$ consists of a four dimensional manifold, on which the map from some open subset of $M_{2 \times 2}'$ is an embedding, and finitely many lower dimensional sets of singularities.
\end{proof}
As alluded to in the beginning of this section, the above analysis, which is
sufficient for our goal is far from being optimal.
Indeed, the map from $M\times M$ to $M_{2\times 2}$
is expected to be an immersion at all points except a one dimension
curve on the diagonal (of points of the form $(m, m)$ and isolated
points outside of the diagonal).

Moreover, the proof shows that in order for $M_{2 \times 2}$ to have
dimension $\leq 3$, $M$ has to be flat ($\det S_m = 0$ for all $m \in M$).
%
\section{The action of \texorpdfstring{$SO(3)$}{SO(3)} on \texorpdfstring{$\BR^3\times \BS^2\times \BS^2$}{R3S2S2}}\label{S:trivializing}
%
Before proceeding to the rotational parallel of property $\star$,
we need to analyze more carefully for which points in $\BR^3\times \BS^2\times \BS^2$
the action of $SO(3)$ is free, and splits nicely. All the results in this
section are either classical results, or immediate consequences of ones.

\begin{dfn}
  Denote by $\gD_{(\BS^2)^{3}}\subset (\BS^2)^{3}$ the set
  of $(p_0, p_1, p_2)$ so that $p_0 = \pm p_1$.
\end{dfn}
\begin{prp}
  The group $SO(3)$ acts freely on
  $(\BS^2)^3 \sm\gD_{(\BS^2)^3}$.
\end{prp}
\begin{proof}
  Since by the assumption, $p_0, p_1$ are not identical or antipodal, $SO(3)$ acts freely
  on the orbit of $(p_0, p_1)$ in $\BS^2\times \BS^2$. The claim follows.
\end{proof}
As we are going to see in a moment, the action of $SO(3)$
is free on a much larger subset
of $(\BS^2)^3$. Nevertheless, removing the set
$\gD_{(\BS^2)^{3}}$ would allow us to represent $(\BS^2)^3$ as
a trivial $SO(3)$-bundle and obtain a section, while removing the
optimal set (namely, $\{(p_0,p_1,p_2) \mid p_0 = \pm p_1 = \pm p_2\}$),
seems to be insufficient for that (see Remark \ref{r:d1-is-optimal}).
\begin{prd}\label{prd:s2cubed_fibers}
  The fibers of
  \[
  \begin{aligned}
    \iota_3 \colon (\BS^2)^{3}&\to [-1,1]^4\\
    (p_0, p_1, p_2)&\mapsto(\langle p_1, p_2\rangle, \langle p_0, p_2\rangle, \langle p_0, p_1\rangle, \det (p_0, p_1, p_2)),
  \end{aligned}
  \]
  are exactly the orbits of $(\BS^2)^3$ under the $SO(3)$ action.
\end{prd}
\begin{proof}
  The proof is standard. Since we are going to use later some of
  the technical details of the proof, we will not omit it.

  Clearly, the value of the function does not change under
  the action of $SO(3)$. Let $\iota_3(p_0, p_1, p_2) = \iota_3(p_0', p_1', p_2')$.
  We would like to find a rotation moving $(p_0, p_1, p_2)$ to $(p_0',p_1',p_2')$.

  Let us first assume that $\alpha = \langle p_0, p_1\rangle \in (-1, 1)$.
  Then, by rotating, we may assume that $p_0 = p_0'$ is the point
  $(1,0,0)$ and $p_1 = p_1'$ is $(\alpha, \sqrt{1 - \alpha^2}, 0)$.

  Denote the inner products
  $\beta= \langle p_0, p_2\rangle$ and $\gamma= \langle p_1, p_2\rangle$.
  Let $p_2 = (x,y,z)$, we get the equations:
  \[\begin{matrix}
   x & = & \beta\\
   \alpha x + \sqrt{1 - \alpha^2} y & = & \gamma\\
   x^2 + y^2 + z^2 &=& 1
  \end{matrix}  \]
  This equations determine $x, y$ and $z^2$ completely (since $\alpha \neq \pm 1$).

  Thus, we obtain two possible
  values for $p_2$, which differ only in the sign of their last coordinate.
  Those two values provides determinants which differ only by sign.
  Moreover, the determinant is zero if and only if $z = 0$ which means that
  there is a unique value for $p_2$.
  Thus, we conclude that under this rotation, the
  value of $p_2$ is determined.

  Let us now deal with the case that $p_1 = \pm p_0$ and $p_2 \neq \pm p_0$.
  In this case, we can rotate and assume without loss of generality
  that $p_0=p_0'=(1,0,0)$ and
  $p_2 = p_2'= (\beta, \sqrt{1 - \beta^2}, 0)$.
  Then, by the inner product relation, we get
  $p_1 = p_1' = (\alpha, 0, 0)$ (where $\alpha = \pm 1$).

  Finally, if $p_1 = \pm p_0, p_2 = \pm p_0$ then after any rotation that makes $p_0$ and
  $p_0'$ to be equal, we get that $p_1 = p_1'$ and $p_2 = p_2'$,
  since the signs are determined by the inner products.
 \end{proof}
\begin{prd}\label{prd:d_s2cubed}
  Let $I_3 \subseteq [-1,1]^{4}$ be the image of $\iota_3$.
  Denote by $D_1$ the image of $\gD_{(\BS^2)^{3}}$ under the map $\iota_3$ from
  \ref{prd:s2cubed_fibers}, then
  $D_1 \subseteq \{-1, 1\} \times [-1, 1]^2 \times \{0\}$.
  Moreover, the $SO(3)$ bundle induced by the map
  \[
  (\BS^2)^{3}\sm \gD_{(\BS^2)^{3}}\to I_3\sm D_1
  \]
  is trivial.
\end{prd}
\begin{proof}
  Using the first case from the proof of proposition \ref{prd:s2cubed_fibers},
  there is a smooth section
  \[s \colon I_3\sm D_1 \to (\BS^2)^3 \sm \gD_{(\BS^2)^3}.\]
  Let us define a map from $SO(3) \times \left(I_3\sm D_1\right) \to (\BS^2)^3$ by
  \[(g, i) \mapsto g \cdot s(i).\]
  Indeed, for every $(p_0, p_1, p_2) \in (\BS^2)^3 \sm \gD_{(\BS^2)^3}$,
  there is a unique rotation $g \in SO(3)$ such that
  $g (s \iota_3 (p_0, p_2, p_2)) = (p_0, p_1, p_2)$.
\end{proof}
Clearly,
\[D_1 = \{(\alpha, \beta, \alpha \cdot \beta, 0) \mid \alpha \in \pm 1, \beta \in [-1,1]\}.\]
Thus, from a practical point of view, it is easy to remove this component from $I_3$.
\begin{rmr}\label{r:d1-is-optimal}
  We have a couple of remarks about the above fibration:

  First, $I_3$ is a three dimensional subset of $[-1, 1]^4$,
  since the determinant $\det(p_0,p_1,p_2)$ is determined
  up to sign from the inner products.
  Similarly, $D_1$ is a one dimensional subspace of $I_3$.

  Second, the choice of $D_1$ is optimal --- there is
  no zero dimensional subset of $I_3$ such that once removed from
  the base of the $SO(3)$-bundle $(\BS^2)^3$, the bundle trivializes.

  This follows from a general classification result of $SO(n)$-bundles,
  by Dold and Whitney, \cite{DW}: Since $(\BS^2)^3\sm\{(p_1,p_2,p_3)|\dim\spn\{p_i\}=1\}$ is
  an $SO(3)$ bundle with three
  dimensional basis, is it fully determined by the zeroth and first $\BF_2$
  cohomology groups of the total space of this bundle -- which are
  computable by a Mayor-Vietoris long exact sequence --- and the zeroth $\BF_2$ cohomology of the
  base. We do not expand on this here as it is
  immaterial for our purposes.
\end{rmr}

\begin{ntt}\label{n:phi_tau_iota}
  Denote by $\phi$ the diffeomorphism:
  \[
  \begin{aligned}
    \phi:\BR^+\times \BS^2 \times (\BS^2)^{2}&\mapsto (\BR^3 \setminus \{0\})\times(\BS^2)^{2}\\
    (a,b,c,d)&\mapsto (ab, c,d).
  \end{aligned}
  \]
  Denote by $\gD_{\BR^3\times(\BS^2)^{2}}\subset \BR^3\times(\BS^2)^{2}$ the union
  of $\{0\}\times (\BS^2)^2$, and the image of
  $\BR^+\times \gD_{(\BS^2)^3}\subset \BR^+\times (\BS^2)^{3}$
  under the map $\phi$.

  Following our convention, we will denote:
  \[
  \begin{aligned}
    \iota_4 \colon (\BR^3 \setminus \{0\}) \times (\BS^2)^{2} & \mapsto  \BR^+ \times I_3 \\
    (a,b,c)&\mapsto (\|a\|, \iota_3(\frac{a}{\|a\|}, b, c)).
  \end{aligned}
  \]
  and
  \[
  \begin{aligned}
    \tau:\BR^3 \times (\BS^2)^{2} \sm \gD_{\BR^3\times(\BS^2)^{2}} & \mapsto SO(3)\\
    (a,b,c)&\mapsto \pi_0(\phi(\frac{a}{\|a\|}, b, c))
  \end{aligned}
  \]

  Finally, we let
  \[\psi \colon SO(3) \times \BR^+ \times (I_3 \sm D_1) \to \BR^3 \times \BS^2 \times \BS^2\]
  be the inverse of $(\tau, \iota_4)$.
\end{ntt}
The lemma below and its corollary are standard facts about the derivative of
the multiplication in a Lie group.
\begin{lma}\label{l:der_product}
  Let $G$ be a Lie group.
  Recall that $g_*:T_e G\to T_g G$ is an isomorphism.
  Then the derivative of the maps
  \[\begin{aligned}
  G\times G&\to G&\qquad G\times G&\to G\\
  (g_1,g_2)&\mapsto g_1g_2&\qquad (g_1,g_2)&\mapsto g_1^{-1}g_2
  \end{aligned}\]
  at the point $({L_{g_1}}_*(h), 0)=({g_1}_*h, 0)$ are given by
  \[
    {R_{g_2}}_*({g_1}_*h),\quad\text{and}\quad
    -{R_{g_1}}_*({R_{g_2}}_*({g_1^{-1}}_*h)=-{R_{g_2}}_*(\Ad_{g_1}h).
  \]
  respectively, where $L, R$ are the left and
  right actions respectively
\end{lma}
\begin{proof}
  Recall that the derivatives
  of the maps
  \[\begin{aligned}
  G\times G&\to G&\qquad G&\to G\\
  (g_1,g_2)&\mapsto g_1g_2&\qquad g&\mapsto g^{-1}
  \end{aligned}\]
  are
  \[\begin{aligned}
      \ &{R_{g_2}}_*({g_1}_* h_1) + L{_{g_1}}_*({g_2}_* h_2)=
      {R_{g_2}}_*({g_1}_* h_1) + {g_1}_*{g_2}_* h_2,\quad\text{ and}\\
        \ & -(R_{g^{-1}})_*((L_{g^{-1}})_*{g_*h})=-{R_{g^{-1}}}_*(h)
  \end{aligned}\]
  respectively, where $((g_1)_*h_1, (g_2)_*h_2)\in T_{(g_1, h_2)}G\times G, g_*h\in T_g G$.
  The first claim follows immediately, whereas the second claim follows from
  the chain rule.
\end{proof}
\begin{cor}\label{c:der_product}
  With the notations of the lemma \ref{l:der_product}
  the maps
  \[
  {R_{g_2}}_*({g_1}_*-):T_e G\to T_{g_1g_2}G\qquad
  -{R_{g_2}}_*\Ad_{g_1}(-):T_e G\to T_{g_2}G
  \]
  are isomorphisms for any $g_1, g_2\in G$,
\end{cor}
\begin{proof}
  Fixing either $g_1$ or $g_2$ we get an isomorphism between the horizontal
  or vertical fiber of $G\times G$ and $G$; of which the above maps are
  derivatives.
\end{proof}
\begin{prp}\label{p:short-exact-sequence:ad-iota}
  Let $p$ be a point in
  $\BR^3 \times \BS^2 \times \BS^2 \sm \gD_{\BR^3 \times \BS^2 \times \BS^2}$.
  The sequence
\[
\begin{tikzcd}
0\to T_e SO(3) \arrow{r}{\psi_*(-{R_{\tau p}}_*\Ad_g(-), 0)} &[4em]
T_p (\BR^3 \times \BS^2 \times \BS^2) \arrow{r}{(\iota_4)_*} &
T_{\iota_4p}(\BR^+ \times I_3)\to 0
\end{tikzcd}
\]
is exact for any $g\in SO(3)$.
\end{prp}
\begin{proof}
This follows from corollary \ref{c:der_product}.
\end{proof}
\begin{rmr}[Where we do and where we do not use the bundle triviality in the rest of the paper]
  The proposition just proven is local in nature, and hence -- as one may
  trivialize the bundle locally -- holds with minor modifications for any
  $G$-bundle, and not only trivial ones. Hence, the proof of lemma
  \ref{lemma:starstar1_equiv_iota_rank} does not depend on the triviality of
  our bundle.

  Moreover, the proof lemma \ref{lemma:rank piX/G rotations} only uses
  ``division along one fiber'', and not the existence of a section to the bundle.

  The place where we do in fact use the bundle triviality is where we, in
  effect, perform a convolution between two copies of the bundle: in
  proposition \ref{p:m4}, and in our application of proposition
  \ref{P:l2fiber}. As can be seen in both proofs, one may convolve
  a trivial bundle with a non trivial one, but not two non trivial ones, and
  in our case we convolve the bundle with itself.
\end{rmr}
%
\section{Geometric aspects II: rotation invariance}\label{S:rotation_inv}
%
There are two major differences between our treatment of the rotations
and our previous treatment in the translations, from Section \ref{s:translations}. While in the case of the action
of $\BR^3$, the normals are not moved and can serve as an anchor
for reconstructing the surface, in the case of the rotations
the invariant parts are defined in a more complicated way.
Moreover, in the case of the translations, we could decompose
the group action on the relevant objects into a part in which
the group acts freely, and a part on which it acts trivially.
This decomposition imposed a structure of trivial group bundle
on the whole space.

In the case of the rotations, the desirable freeness
of the group action and, more importantly, the triviality of the corresponding
bundle are no longer true at every point. Instead, we
have to remove a subspace of co-dimension 2 in order
to trivialize the bundle, as we did in Section \ref{S:trivializing}.
This operation is, from the
perspective of Sections \ref{S:rep} and \ref{S:functional_analysis},
immaterial, as it happens on a measure zero in the surface (specifically,
see proposition \ref{P:l2fiber}, corollary \ref{C:l2fiber} and Remark
\ref{R:l2fiber}).

As in the case of the translations, we need to isolate a property of the
manifold $M$, measuring its self similarity under the action of the group.

For the ease of notation, denote
\[\RSS = \BR^3 \times \BS^2 \times \BS^2\]

\begin{ntt}
  Let $G = SO(3)$. Denote the action of $g\in G$ on $\RSS$
  by $\eta_{g,\RSS}$.
   We define two 7-folds and one 4-fold in $G\times \RSS$:
  \[
  \begin{aligned}
    Y_7=Y_7(M):=& G\times M_{2\times 2},\\
    Z_7=Z_7(M):=& \{(g, g\ol{m}) \mid g \in G,\, \ol{m} \in M_{2\times2}\},\\
    X_4=X_4(M):=& Y_7\cap Z_7.
  \end{aligned}
  \]
  Note that as $SO(3)$ is 3-dimensional, if $M_{2\times2}$ is 4-dimensional,
  then both $Y_7$ and $Z_7$ are 7-dimensional, whereas the ambient space
  $G\times \RSS$ is 10-dimensional.
  As in the translatory case, $TY_7 \cong TG \oplus TM_{2\times2}$ globally,
  while the same decomposition holds only locally for $Z_7$.
  Denote the inclusion $X_4\to G\times \RSS$ by
  $i_{X_4/G\times \RSS}$.
  Denote the composition of the inclusion map and the projection maps from $X_4$
  to the first and second components in $G\times \RSS$ by
  $\pi_{X_4/G}, \pi_{X_2/\RSS}$,
  respectively. Since $\range(\pi_{X_4/\RSS})\subset M_{2\times 2}$
  we may also define
  \[
  \pi_{X_4/M_{2\times 2}}:=i^{-1}_{M_{2\times 2}/\RSS}\circ
  \pi_{X_4/\RSS}.
  \]
  Finally we denote
  \[
  \rho_{X_4/M_{2\times 2}}:=\eta_{\pi_{X_4/G}(-)^{-1},\RSS}\circ
  \pi_{X_4/M_{2\times 2}}
  =i^{-1}_{M_{2\times 2}/\RSS}\circ
  \eta_{\pi_{X_4/G}(-)^{-1},\RSS}\circ\pi_{X_4/\RSS}
\]
  where the composition in the right term of the last equality is well defined by
  the definition of $X_4$.
\end{ntt}
As in the case of $X_2$, the notation of $X_4$ suggests that it would be
a four dimensional manifold. This is certainly not true in general.
\begin{exm}[The unit sphere running example becomes a non-example]
Let $M = \BS^2$, which is invariant under $SO(3)$,
then, $X_4 = Y_7 = Z_7$ and it is 7-dimensional.
\end{exm}
As before, the set $\{id\} \times M_{2 \times 2}$
is contained in $X_4$, forcing its dimension to be
at least the dimension of $M_{2 \times 2}$ which is,
at most cases, $4$. Nevertheless, $X_4$ does not have to be a manifold, in the
sense that it might be composed of several manifolds
of various dimensions, with non-trivial intersections.

The following definition, which is a strengthening of $(\star)$, isolates
a sufficient condition for $X_4$ to behave nicely in order for our
reconstruction method to apply.
\begin{dfn}[Property $\star\star$]\label{d:property_star_star}
  Let $M$ be a twice differentiable surface satisfying $\star$. We say that $M$
  satisfies $\star\star$ if there is a subset $X'_3 \subseteq X_4$ of
  $\mathcal{I}$-dimension at most 3 such that at points
  \[
  p = (g, (m_2 - m_1, N_{m_1}, N_{m_2}))
  = (g, (g(m_2' - m_1'), gN_{m_1'}, gN_{m_2'})))\in X_4\sm X'_3,
  \]
the following properties hold:
\begin{enumerate}
\item\label{i:iso_so3}
  The map $(\iota_4 \restriction M_{2\times 2})_*|_{\bar{m}}$
  is of rank $4$ for $\bar{m} = (m_2 - m_1, N_{m_1}, N_{m_2})$.
\item\label{i:shape_so3}
  Let $S_r$ be the shape operator of $M$ at point $r$. Then
  \[\rank (g_* S_{m_1'} - S_{m_1}) =
  \rank (g_* S_{m_2'} - S_{m_2}) = 2.\footnote{for a linear map $S$ on $V$ and $g \colon V \to U$, $g_*(S)$ is a linear map on $U$ defined by $g_*(S)(x) = g S g^{-1}(x)$.}
\]\end{enumerate}
\end{dfn}
\begin{lma}\label{lemma:rank piX/G rotations}
  For points $p \in X_4 \sm X'_3$, which do not project to $D_1$ under $\iota_4$ and do not project to $\gD_{\BS^2 \times \BS^2}$ under the corresponding projection, we have
  $\rank d \pi_{X_4/G} |_p = 3$.
\end{lma}
\begin{proof}
  We use the notations of definition \ref{d:property_star_star}, and
  similarly to the translatory case (see lemma \ref{lemma:rank piX/G translations}), consider the map
  \[
  (\pi_{X_4/G},\pi_{X_4/M_{2\times 2}},\rho_{X_4/M_{2\times 2}}):X_4\to G\times M_{2\times 2}
  \times M_{2\times 2}.
  \]
  Let $p \in X_4 \sm X_3'$ which does not project to $D_1$, and let
  \[
  (g,\quad m_1 - m_2, N_{m_1}, N_{m_2},\quad
  m_1' - m_2', N_{m_1'}, N_{m_2'})  \in G \times M_{2\times2} \times M_{2\times2},
  \]
  be the image of $p$, where $m_1, m_2, m_1', m_2' \in M$. Since our point $p$ is
  not in the preimage of $D_1$, we conclude that there is an open neighborhood of
  $(m_1, m_2, m_1', m_2') \in M^4$ and a natural smooth map from it to $G \times M_{2 \times 2} \times M_{2 \times 2}$, sending $(r_1, r_2, r'_1, r'_2)$ to
  \[
    (\tau(r_1 - r_2, N_{r_1}, N_{r_2}) \tau(r_1' - r_2', N_{r_1'}, N_{r_2'})^{-1} , 
    \quad r_1 - r_2, N_{r_1}, N_{r_2},\quad
    r_1' - r_2', N_{r_1'}, N_{r_2'})
  \]
  This point is in the image of $X_4$ if and only if there is $h \in G$ such that:
  \[
    r_1 - r_2 = h(r_1' - r_2'),\quad 
    N_{r_1}= h(N_{r_1'}),\quad 
    N_{r_2}= h(N_{r_2'})
    \]

  By the freeness of the action of the group, $h$ is unique:
  \[h = \tau(r_1 - r_2, N_{r_1}, N_{r_2}) \tau(r_1' - r_2', N_{r_1'}, N_{r_2'})^{-1}.\]

  Let us analyze which tangent vectors in $M^4$
  correspond to the kernel of
  $(\pi_{X_4 / G})_*|_p$. As we will see in the next lemma, \ref{lemma:starstar1_equiv_iota_rank}, those
  points correspond to points in $X_3'$.
  Let $x_1,x_2,x_1',x_2'$ be tangent vectors
  at the corresponding
  points $m_1,m_2,m_1',m_2'$ and let $g$ be the value of the rotation.
  Take the derivatives of the above equations,
  assuming that $x_1, x_2, x_1', x_2'$ correspond
  to a tangent vector in $X_4$
  with trivial projection to $G$ (which means that the
  derivative of $h$ is zero at that direction).
  We obtain the following equations:
  \[
    x_1 - x_2 = g_*(x_1' - x_2'),\quad
    S_{m_1}(x_1) = g_*(S_{m_1'}(x_1')),\quad
    S_{m_2}(x_2) = g_*(S_{m_2'}(x_2'))
    \]
  In order for those equations to make sense,
  we canonically embed all tangent vectors to $\BR^3$.

  Let
  \[t = x_1 - g_* x_1' = x_2 - g_* x_2' \in T_{m_1} M \cap T_{m_2} M.\]
  Since this vector is perpendicular to both
  $N_{m_1}$ and $N_{m_2}$, which are
  linearly independent, it lies in a one dimensional
  subspace of $\BR^3$.

  By rearranging the equations we obtain:
  \[
  (g_* S_{m_1'} - S_{m_1})(g_* x_1') = S_{m_1} (t),\quad \\
  (g_* S_{m_2'} - S_{m_2})(g_* x_2') = S_{m_2} (t)
  \]
  Using requirement \ref{i:shape_so3} in definition
  \ref{d:property_star_star},
  we conclude that $x_1, x_2, x_1', x_2'$ are
  uniquely determined by $t$, which
  means that the linear space of such tangent
  vectors is at most one dimensional.

  Thus, $\dim \ker \pi_{X_4/G} \leq 1$, or equivalently
  $\rank \pi_{X_4/G} \geq 3$, as claimed.
\end{proof}

The proof of the above lemma
can be adjusted to the case $N_{m_1}= \pm N_{m_2}$,
to get $\rank \pi_{X_4/G} \geq 2$.

Recall that $\iota_4$ is a differentiable map from
$(\RSS) \sm \Delta_{\RSS}$
to $\mathbb{R}^+ \times (I_3 \sm D_1)$.

\begin{lma}\label{lemma:starstar1_equiv_iota_rank}
  If $x$ is a point in $X_4$, corresponding to
  $p, p' \in M_{2\times 2}$, so that
  $p'\in\iota_4^{-1}(\mathbb{R}^+ \times (I_3 \sm D_1))$. Let us assume further
  that the dimension of $M_{2 \times 2}$ around $p, p'$ is $4$, and that
\[\rank (\iota_4 \restriction M_{2 \times 2})_* |_p
= \rank (\iota_4 \restriction M_{2 \times 2})_* |_{p'} = 4.\]
Then $\rank (\pi_{X_4/M_{2\times 2}})_*|_x = \rank (\rho_{X_4/M_{2\times 2}})_*|_x = 4$.
\end{lma}
This lemma is the rotational equivalent of lemma \ref{L:prop_star1_is_regualrity}
\begin{proof}
We consider the derivatives of the maps
\[\begin{aligned}
\pi_{X_4/M_{2\times 2}}&=i^{-1}_{M_{2\times 2}/\RSS}
\circ i_{X_4/\RSS},\\
\rho_{X_4/M_{2\times 2}}&=i^{-1}_{M_{2\times 2}/\RSS}\circ
\eta_{\pi_{X_4/G}(-)^{-1},\RSS}\circ i_{X_4/\RSS}
\end{aligned}\]
at $T_xX_4$.

Let $v$ be a tangent vector in $T_xX_4$. We set notations to correspond with those
of lemma \ref{l:der_product}: $h:=(\pi_{X_4/G})_*(v), g_1:=\pi_{X_4/G}(x),g_2:=\tau(p)$. Note that $g_1 \tau(p') = g_2$. In particular,
\[
p'= \rho_{X_4/M_{2\times 2}}(x) = \big(\eta_{\pi_{X_4/G}(x)^{-1},\RSS}\circ i_{X_4/\RSS}\big)(x)
=\psi\big(\eta_{g_1^{-1},\RSS} \tau(p), \iota_4(p)\big).
\]

The vector $v$ is in $\ker (\pi_{X_4/M_{2\times 2}})_*$, if and only if
$\tau_*(\pi_{X_4/M_{2\times 2}})_*(v)) = 0$ and $(\iota_4)_*(\pi_{X_4/M_{2\times 2}})_*(v) = 0$,
which by lemma \ref{l:der_product}, implies that
\[
w:=(\rho_{X_4/M_{2\times 2}})_*(v) = \big((\eta_{\pi_{X_4/G}(-)^{-1},\RSS}\circ i_{X_4/\RSS})_*
\restriction T_xX_4\big) (v)
=\psi_*\big((-R_{g_2})_*\Ad_{g_1}(h)\big).
\]
This vector is trivial if and only if $h={\pi_{X_4/G}(x)}_*(v)$ is trivial if and
only if $v$ is trivial. Moreover it has a trivial
$(\iota_4)_*$ image since $\eta$ acts through $\iota_4$ fibers.
By proposition \ref{p:short-exact-sequence:ad-iota} vectors of the form
of $w$ are exactly the vectors in the kernel on $(\iota_4)_*$.
\end{proof}
\begin{ntt}\label{n:m4}
  Let $M\subset\BR^3$ be a twice differentiable compact surface.
  Denote by $M_{2\times 4}'\subset SO(3)\times(\RSS \sm \gD_{\RSS})^2$
  the set
  \[
  \
  \{(g, \ol{m}_1, \ol{m}_2) |\ol{m}_1,\ol{m}_2\in M_{2\times2} \sm \gD_{\RSS},\
  \tau(\ol{m}_1) g = \tau(\ol{m}_2)\}.
  \]
  Denote by
  \[
  M_{2\times4}\subset SO(3)\times(\BR^+\times (I_3 \sm D_1))^2
  \]
  the image of $M_{2\times4}'$ under the map $(id_{SO(3)}, \iota_4, \iota_4)$.
  Finally, denote by
  \[
  \ti{M}_{2\times 4}\subset\RSS\times \BR^{+} \times (I_3\sm D_1)
  \]
  the image of $M_{2\times 4}$ under the map
  $(g, i_1, i_2) \mapsto (\psi(g, i_2), i_1)$.
\end{ntt}

Each $\ol{m}_1 \in M_{2 \times 2}$, contributes to $\ti{M}_{2\times 4}$
a copy of $M_{2 \times 2}$ rotated by $\tau(\ol{m}_1)^{-1}$.   Moreover,
using requirement \ref{i:iso_so3} of $(\star\star)$, and the arguments
from lemma \ref{L:gauss_map_finite}, we conclude that besides a
set of positive co-dimension in $\BR^+ \times I_3$, each value in the
image of $\iota_4$ corresponds to {\it finitely} many rotations.

Thus, similarly to proposition \ref{p:m2}, by concentrating on a
fiber, we get finitely many rotated copies of $M_{2 \times 2}$, and the
challenge is to show that those copies do not intersect too much.

\begin{prp}\label{p:m4}
  Suppose $M$ satisfies property $\star\star$.
  Then there is an open subset
  \[
  U\subseteq \BS^2\times \BR^+ \times (I_3 \sm D_1)
  \]
  which is dense and of full relative measure in the set of possible fibers,
  so that for a point $f \in U$, the intersection
  \[
  \ti{M}_{2\times 4} \cap \BR^3\times \BS^2 \times \{f\},
  \]
  is the union $(g_1\circ M \cup \dots \cup g_n\circ M) \times\{f\}$,
  for some $g_i \in SE(3)$, where the copies of
  $M$ intersect only at finitely many points,
  along the pullback of the diagonals.
  Moreover, $\{g_1, \dots, g_n\}$ are continuous in $f$,
  in the same way as in proposition \ref{p:m2}.
\end{prp}
\begin{proof}
  The argument is similar to the proof of proposition \ref{p:m2}.
  As in the proof for the translational case, $U$ is clearly open,
  and we going to show that its complement has a positive $\mathcal{I}$-co-dimension.

  Let
  \[
  \ol{m}_1,\ldots, \ol{m}_4\in M_{2\times2}',\quad g_{1,2},g_{3,4}\in SO(3)
  \]
  be such that $\iota_4\ol{m}_i\not\in\BR^+\times D_1$, and such that
  denoting the image of  $\ol{m}_i$ in  $M_{2\times 2}$ by $\ti{m}_i$ the
  points
  \[
    (g_{1,2}, \ti{m}_1, \ti{m}_2),\quad
    (g_{3,4}, \ti{m}_3, \ti{m}_4) \quad \in M_{2\times 4}'
  \]
  are both mapped to the same point in $M_{2\times 4}$. Note
  that $\iota_4\ol{m}_2,\iota_4\ol{m}_4\not\in\BR^+\times D_1$
  by the choice of $f$, whereas {\em it is not the case}
  for $\ol{m}_1, \ol{m}_3$. We will get rid of this assumption on
  $\ol{m}_1, \ol{m}_3$ soon.

  By the definition of $M_{2 \times 4}$, this means that
  \begin{gather*}
    \iota_4(\ol{m}_1) = \iota_4(\ol{m}_3) \\
    \iota_4(\ol{m}_2) = \iota_4(\ol{m}_4) \\
    \tau(\ol{m}_1)^{-1} \tau(\ol{m}_2)  = g_{1,2}=g_{3,4}=\tau(\ol{m}_3)^{-1} \tau(\ol{m}_4).
   \end{gather*}
  Rearranging, denoting $g:=\tau(\ol{m}_4) \tau(\ol{m}_2)^{-1}$,
  and using $\psi$ we get:
  \begin{equation}\label{e:g_from_m1_m3}
   \begin{gathered}
     g =  \tau(\ol{m}_4) \tau(\ol{m}_2)^{-1} =
     \tau(\ol{m}_3) \tau(\ol{m}_1)^{-1},\\
     g\ol{m}_1=\tau(\ol{m}_3)\tau(\ol{m}_1)^{-1}
     \psi(\tau\ol{m}_1,\iota_4\ol{m}_1)=\psi(\tau{m}_3,\iota_4\ti{m}_3)
      =  \ol{m}_3,\\
      \text{and similarly }\quad g \ol{m}_2 = \ol{m}_4
   \end{gathered}
  \end{equation}

  Namely, $x_1 = (g, \ol{m}_3) = (g, g\ol{m}_1) \in X_4$ and similarly,
  $x_2 = (g, \ol{m}_4) = (g, g\ol{m}_2) \in X_4$.

  Removing the assertion $g = \tau(\ol{m}_3) \tau(\ol{m}_1)^{-1}$, and
  adding $g$ as an unknown, the above
  system of equations makes sense for arbitrary $\ti{m}_1, \ti{m}_3$, and we
  can remove the hypothesis that $\iota_4(\ol{m}_1), \iota_4(\ol{m}_3) \notin D_1$.

  We now split into cases, based on whether $x_i \in X_3'$ or
  not, for $i = 1,2$, and finally treat the diagonal case.
  For each case $c$, we construct a closed set
  $F_c \subseteq \BS^2 \times \BR^{\geq 0} \times I_3$,
  of positive $\mathcal{I}$-co-dimension. Eventually, we let
  \[U = \BS^2 \times \BR^+ \times I_3 \sm \bigcup_c F_c.\]

  Let us remark that there is a difference between the roles
  of $x_1$ and $x_2$, since the $I_3$ part of $f$ is $\iota_4(\ol{m}_2)$.
  This distinction was forced in the definition of $\ti{M}_{2\times 4}$.

  {\em Case 1:} $(x_1, x_2) \in X_4 \times X_3'$. The set
  $\bar{F}_1 = \iota_4 \circ \pi_{Y_7/M_{2\times 2}} \circ i_{X_4/Y_7}(X_3')$
  is the image of a set of $\mathcal{I}$-dimension $\leq 3$.
  Thus, $F_1 = \BS^2 \times \bar{F}_1$ is a subset
  of $\BS^2\times \mathbb{R}^+ \times I_3$ of $\mathcal{I}$-co-dimension at least 1.

  {\em Case 2:} $(x_1, x_2) \in X'_3 \times (X_4 \sm X'_3)$.
  Consider for each
  $x = (g, \ol{m}_3) \in X'_3$, the collection
  of all $y = (g', \ol{m}_4) \in X_4\sm X'_3$ that correspond to $x$.
  Per equation \ref{e:g_from_m1_m3},
  $g=g'$ and it is completely determined from $\ol{m}_2, \ol{m}_4$ (as
  $\iota_4(x_2)\not\in\BR^+\times D_1$).
  By lemma \ref{lemma:rank piX/G rotations},
  the Jacobian of the equation
  $g = g'$ is of rank 3, and therefore
  for each $x \in X'_3$, the set of solutions
  in $X_4 \sm X'_3$ is (at most) one dimensional.
  Let us decompose $X_3'$ into a countable collection of smooth images
  of the closed 3-dimensional Euclidean ball. For each such component,
  by the rank of the Jacobian, we obtain a 4-dimensional manifold of solutions.

  Let $F_2$ be the projection of this collection to $\BS^2 \times \BR^+ \times I_3$, then it is a set of $\mathcal{I}$-dimension $\leq 4$.

  {\em Case 3:} $(x_1, x_2) \in (X_4 \sm X_3') \times (X_4 \sm X_3')$ and $N_{m_1} \neq N_{m_2}$, $N_{m_3} \neq N_{m_4}$.
  For each $i\in\{1,2,3,4\}$ let $m_{2i-1}, m_{2i}$ be
  the points on $M$ so that
  \[
  \ol{m}_i=(m_{2i-1}-m_{2i}, (m_{2i-1}, N_{m_{2i-1}}),(m_{2i}, N_{m_{2i}})).
  \]
  Let us pick a chart $U\subset\BR^2$ and a chart-map $f\colon U\to M$,
  so that $m_i\in f(U)$.
  Let $p_1,\ldots,p_8\in U$ be eight points so that $f(p_i)=m_i$.
  Then, substituting in Equation \ref{e:g_from_m1_m3}, we see that
  the set of points in $M_{2\times4}$ for which $M_{2\times4}'\to M_{2\times4}$
  is not one to one---and
  where all eight relevant points of $M$ are in $U$---are in one to one
  correspondence
  with the $f$ images of the solutions set of
  \begin{equation}
  \begin{matrix}\label{e:so3cond}
  g_{1,2} = g_{3,4}, &\\
  g(m_{1}-m_{2})=m_{5}-m_{6},& \\
  g(m_{3}-m_{4})=m_{7}-m_{8},& \\
  g\frac{\nabla f(q_{i})}{\|\nabla f(q_{i})\|}=\frac{\nabla f(q_{i + 4})}{\|\nabla f(q_{i + 4})\|},&  \text{ for }i \in \{1,2,3,4\}
  \end{matrix}
  \end{equation}
  for $(q_1, \ldots, q_8)\in U^8$ and $g$ is defined as above.

  Instead of solving the above equation in $U^8$, we are going to ignore the
  first equation and solve the system of equations in $SO(3) \times U^8$ (treating
  $g$ as an unknown).
  Any solution for the original system of equations provides a solution to the
  modified system, and thus bounding the set of solution for the modified
  system suffices.

  Let us compute the Jacobian matrix of this map around $p_1,\ldots,p_8$
  with respect to the local coordinates
  $x_i, y_i$ about the $p_i$-s. We ignore the $SO(3)$-coordinate and
  concentrate on the remaining 18 rows by 16 columns sub-matrix,
  where below we follow the following conventions:
  \begin{itemize}
  \item each entry stands for a 3 rows by 2 columns sub matrix, and
  \item the order of columns is $q_1,q_2,q_5,q_6,q_3,q_4,q_7,q_8$; i.e.\
    the functions on these columns are evaluates at $p_1,p_2,p_5,p_6,p_3,p_4,p_7,p_8$
    respectively, and
  \item the order of the rows is adjusted to form the block matrix:
   \end{itemize}
  {\tiny
  \[\begin{pmatrix}
  g\frac{\partial N_{m_1}}{\partial x_1, y_1}&0_{3\times 2}&
  -\frac{\partial N_{m_5}}{\partial x_5, y_5}&0_{3\times 2}&
  0_{3\times 2} &0_{3\times 2} &0_{3\times 2} &0_{3\times 2}\\
  0_{3\times 2} &g\frac{\partial N_{m_2}}{\partial x_2, y_2}&
  0_{3\times 2} &-\frac{\partial N_{m_6}}{\partial x_6, y_6}&
  0_{3\times 2} &0_{3\times 2} &0_{3\times 2} &0_{3\times 2}\\
  g\frac{\partial f}{\partial x_1, y_1}&
  -g\frac{\partial f}{\partial x_2, y_2}&
  -\frac{\partial f}{\partial x_5, y_5}&
  \frac{\partial f}{\partial x_6, y_6}&
  0_{3\times 2} &0_{3\times 2} &0_{3\times 2} &0_{3\times 2}\\
  0_{3\times 2} &0_{3\times 2} &0_{3\times 2} &0_{3\times 2} &
  g\frac{\partial N_{m_3}}{\partial x_3, y_3}&0_{3\times 2}&
  -\frac{\partial N_{m_7}}{\partial x_7, y_7}&0_{3\times 2}\\
  0_{3\times 2} &0_{3\times 2} &0_{3\times 2} &0_{3\times 2} &
  0_{3\times 2} &g\frac{\partial N_{m_4}}{\partial x_4, y_4}&
  0_{3\times 2} &-\frac{\partial N_{m_8}}{\partial x_8, y_8}\\
  0_{3\times 2} &0_{3\times 2} &0_{3\times 2} &0_{3\times 2} &
  g\frac{\partial f}{\partial x_3, y_3}&
  -g\frac{\partial f}{\partial x_4, y_4}&
  -\frac{\partial f}{\partial x_7, y_7}&
  \frac{\partial f}{\partial x_8, y_8}\\
  \end{pmatrix}.\]
  }
  As in the proof of proposition \ref{p:m2} we may
  assume w.l.o.g.\ (via local change of coordinates),
  that $p_i=(x_i, y_i)$ are such that:
  \[
  g \frac{\partial f}{\partial x_i, y_i}|_{p_i}=
  \frac{\partial f}{\partial x_{i + 4}, y_{i + 4}}|_{p_{i + 4}}.\]

  Similarly, the differentials of the normals are simply the shape operators under those coordinates.

  Adding the column 1,2 to 3,4 and 5,6 to 7,8, eliminating (the trivial)
  rows 1,4 and rearranging
  we obtain a diagonal block matrix, with two blocks of the form ($i = 1, 3$):
  \begin{equation}\label{e:submatrix_rotation}
   \begin{pmatrix}
    g_* S_{m_i} - S_{m_{i+4}} & 0_{2\times 2} & -S_{m_{i+4}} & 0_{2\times 2} \\
    0_{2\times 2} & g_* S_{m_{i + 1}} - S_{m_{i + 5}} & 0_{2\times 2} & -S_{m_{i + 5}}  \\
	0_{3\times 2} & 0_{3\times 2} & -T_{m_{i + 4}} & T_{m_{i + 5}}
    \end{pmatrix}
  \end{equation}
  where $T_{m_{i + 4}}, T_{m_{i + 5}}$ are matrices of rank 2.
  If we assume further that $N_{m_{i + 4}} \neq N_{m_{i  + 5}}$,
  then the rank of the corresponding block is 3.

  Thus, in the non-diagonal case,
  the rank of this matrix is at least
  \[3 + \rank(g_* S_{m_1} - S_{m_5}) + \rank(g_* S_{m_2} - S_{m_6}).\]
  Similarly, the rank of the second block is at least
  \[3 + \rank(g_* S_{m_3} - S_{m_7}) + \rank(g_* S_{m_4} - S_{m_8}).\]
  Recall that
  \[
  (g, \ti{m}_3) = (g, g\ti{m}_3),\quad (g, \ti{m}_4) = (g, g\ti{m}_2)
  \in X_4\sm X'_3.
  \]
  whence, by Requirement \ref{i:shape_so3} in $\star\star$, we conclude that the
  rank of the Jacobian matrix is at least 14.
  Thus, this set contributes a countable collection of sub-manifolds of
  dimension $\leq 19 - 14 = 5$. Let $F_3$ be its projection
  to $\BS^2 \times \mathbb{R}^+ \times I_3$,
  then its $\mathcal{I}$-dimension is $\leq 5$.

  {\it Case 4:} the diagonal case, $N_{m_1} = N_{m_2}$. Recall that $f$ is composed
  from a normal $s \in \BS^2$ and an element $i \in \BR^+ \times I_3$.

  For each $i \in \BR^+ \times I_3$ that does not belong to the projection
  of $X_3'$, there are only finitely many $\ol{m} \in M_{2\times 2}$ such
  that $\iota_4(\bar{m}) = i$. In particular, the set
  $G_i = \{\tau(\ol{m}) \mid \ol{m} \in M_{2\times 2},\, \iota_4(\ol{m}) = i\}$
  is finite.

  Let $F_4$ be the set of all
  $f = (s, i) \in \BS^2 \times \BR^+ \times I_3$ such
  that either $i$ belongs to the projection of $X_3'$ or
  $s \in G_i X_1'$. Since everything is smooth, one can decompose the set
  $\bigcup \{\{i\} \times G_i \mid i \in \BR^+ \times I_3\}$ into countably
  many smooth functions from
  $\range \iota_4 \restriction M_{2\times 2} \setminus \bar{F}_1$ to $SO(3)$.
  Therefore $F_4$ is of $\mathcal{I}$-dimension $\leq 5$.

  Since $M$ satisfies $\star\star$, it satisfies $\star$ as well.
  Applying lemma \ref{L:gauss_map_finite}, we conclude that
  for $f \notin F_4$, the Gauss map is finite around points
  with $N_m = N_{m_2}$, thus concluding the proof.
\end{proof}
%
\section{Representation theoretic aspects}\label{S:rep}
%
The object of this section is give an explicit representation theoretic
construction of the Fourier coefficients of the distribution
$\gd_{\ti{M}_{2\times4}}$. As alluded in the toy problem in the beginning of Section
\ref{S:results} we simply give a representation theoretic analog of the
construction of $M_{2\times 2}$ and $\ti{M}_{2\times4}$; thus connecting the
geometric construction to the functions $\rho_{n,m}$ (see \ref {d:rho}) and the
coefficients  $f^{\mu}_{d, n,m,n',m'}$ and $F^{\mu}_{d, d'}$ (see \ref{d:fF}).
An important stepping stone in this direction is to generalize the following
classical theorem:
\begin{thm}[Moments of self convolution, essentialy due to De Moivre]\label{t:con_moments}
  Let $f$ be a finite distributions on $\BR^n$, let
  $\ti{f}:=f\circ (x\mapsto-x)$,
  and denote by $\cM_-$ the moment
  generating function, then $\cM_{f*\ti{f}}=\cM_f\cdot\cM_{\ti{f}}$.
\end{thm}
Note that for every $g \in \BR^n$, $f * \ti{f} = \eta_{g, \BR^n} f * \eta_{g, \BR^n} \ti{f}$, where $\eta$ denotes the group action. So, the construction is
invariant under the action of the group of translations.
\subsection*{Group invariants of self convolution}
Throughout this subsection, $G$ is a compact Lie Group.
\begin{prp}\label{P:todo1}
  Let
  $E\to B, E'\to B'$ two $G$ bundles.
  Then $E\times E'$ is a $G$ bundle under the diagonal action, where the base
  of the bundle --
  $(E\times E')/ G$, is itself a $G$ bundle over $B\times B'$.
\end{prp}

\begin{prp}\label{P:l2fiber}
  Let $G$ be a Lie group with a Haar measure $\mu$.
  Let $\pi\colon E\to B$ be a $G$ bundle, let $\pi^*\colon L^2(B)\to L^2(E)$ be
  the induced pullback map. Then $(L^2(E))^G=\pi^* L^2(B)$, and
  the map
  \[
  \begin{aligned}
    \pi_*:L^2(E)&\to L^2(B)\\
    f&\mapsto \left(b\mapsto\int_{\pi^{-1}(b)}f d\mu\right)
  \end{aligned}
  \]
  is well defined; Moreover if $\mu(G)=1$, then $\pi_*$ is left inverse of
  $\pi^*$.
\end{prp}

\begin{cor}\label{C:l2fiber}
  With the notations of the proposition above we have
  \[
  (L^2(E)\otimes L^2(E))^G=L^2(E\times E)^G=L^2((E\times E)/G).
  \]
\end{cor}
\begin{rmr}\label{R:l2fiber}
  With the notations of the propositions and corollary above, assume that
  $i\colon B_0\subset B$ is such that the $B\sm B_0$ is of measure zero, and $E_0$ --
  the restriction of $E\to B$ is a trivial bundle. Choosing a trivialization
  to $E_0\to B_0$, we have an isomorphism
  \[
  \begin{aligned}
    (E_0\times E_0) / G&\cong \left((B_0\times G)\times (B_0\times G)\right) / G\\
    =&
    B_0\times B_0 \times ((G\times G) / G)\cong B_0\times B_0\times G,
    \end{aligned}
  \]
  where the $G$-action on the product spaces is diagonal,
  and the last copy of $G$ above can be identified with the
  anti-diagonal.
  We stress that the trivialization of this bundle depends (in a very explicit
  way) on the trivialization we choose to $E_0\to B_0$.
\end{rmr}

\subsection*{Representation theoretic interpretation of the construction}
\begin{prd}[Wigner D-matrices and some of their properties]
  Wigner $D$-function matrices are a basis for $[j]\otimes[j]^\vee$
  for each $j$.
  The matrix $\left(D^j_{m'm}(\ga,\gb,\gc)\right)_{m'm}$ is the
  representation of the $z-y-z$ rotation
  $\mathcal{R}=\mathcal{R}(\ga,\gb,\gc)$, using the basis of the $Y_{jm}$, namely
  \[
  \left(Y_{jm'}(\mathcal{R}r)\right)_{m'}=\left(D^j_{m'm}(\ga,\gb,\gc)\right)_{m'm}\left(Y_{jm}(r)\right)_m.
  \]

   By Peter-Weyl theorem, the $D^j_{m'm'}$s form an orthogonal basis of $L^2(SO(3))$.
  Moreover, their norms are given by $\sqrt{\frac{8}{2j+1}}\pi$.
  The Kronecker product formula gives the identity
  \[
  D^j_{mk}D^{j'}_{m'k'}=\sum_{J=|j-j'}^{j+j'}
  \langle jmj'm'|J(m+m')\rangle\langle jkj'k'|J(k+k')\rangle
  D^J_{(m+m')(k+k')},
  \]
  where $\langle|\rangle$ are the Clebsch-Gordan coefficients.
  Finally, the pullback of the Hopf fibration $SO(3)\to \BS^2$
  (coming the action of rotation on the point $(1, 0, 0)\in \BS^2$) is
  given by $D^j_{m0}(\ga,\gb,\gc)=\sqrt{\frac{4}{2j+1}}Y_{jm}^*(\gb,\ga)$.
\end{prd}
\begin{dsc}\label{d:rep_theory_interpretation}
  Let $\mu$ be a distribution over $\mathbb{R}^3 \times \BS^2$.

  Recall the definitions of $\rho_{n,m}(\mu), f^{\mu}_{d,n,m,n',m'}$ and $F^{\mu}_{d,d}$
  from definitions \ref{d:rho} and \ref{d:fF}.

  Clearly, any distribution is uniquely determined by its moments,
  as they correspond to a dense collection of test functions.
  We call the coefficients in this
  representation {\em the moments of $\mu$}.

  We are particularly interested in distributions which are the delta functions
  of a surface. Those distributions are positive and
  by Riesz–Markov–Kakutani representation theorem, the other direction is true as well:
  if such a tempered distribution is positive then it corresponds to a positive Radon measure.
  Thus, a surface is completely determined by the moments of its delta
  function, and by the representation theorem, this dependency is continuous in
  the coefficients.

  By the Peter-Weyl theorem, $L^2(\BS^2)$ is the closure of the
  direct sum $\oplus_{n=0}^\infty[n]$,
  where $[n]$ is the irreducible representation of $SO(3)$ of dimension
  $2 \cdot n + 1$;
  likewise, $L^2(SO(3))$ is the closure of the direct sum
  $\oplus_{n=0}^\infty([n]\otimes[n]^{\vee})$.
  This representation provides a standard orthonormal basis for
  $L^2(SO(3))$. So, we use the term {\it coefficients} to refer
  to the coefficients of the representation of
  a distribution $\mu$ on $L^2(SO(3))$ as an infinite sum, using this basis.

  Finally, for a distribution in product spaces such that we
  already fixed a basis to each one of them, we use the term {\it coefficients}
  to denote the coefficients in tensor product of the spaces,
  taking the product of the bases.
\end{dsc}
\begin{ntt}
  Let us denote by $\iota_{\mathbb{R}^3}(f) = f * \ti{f}$.
\end{ntt}
\begin{lma}
Let $\mu$ be a distribution over $\BR^3 \times \BS^2$.
The coefficients of degree $\leq d$ of $\iota_{\BR^3}(\mu)$ are quadratic polynomials
at the coefficients of degree $\leq d$ of $\mu$.
\end{lma}
\begin{proof}
Since the convolution in the definition of $\iota_{\BR^3}$ is defined fiber-wise on $\BS^2$, it is clear that we just need to address the $\BR^3$ part.
Using theorem \ref{t:con_moments}, and the fact that $\mathcal{M}_\mu$ is a formal power series, the result follows.
\end{proof}
\begin{rmr}
Let $M$ be a compact twice differentiable surface, embedded
into $\BR^3\times \BS^2$. By the definition, $M_{2\times 2}$
is the support of $\iota_{\BR^3}(\gd_M)$. Moreover, if $M$ satisfies
$\star$, then the self intersections and the degeneracy locus are
relatively null, $\iota_{\BR^3}(M) = \gd_{M_{2\times 2}}$ .
\end{rmr}
Note that since the spaces $V_d$ are $SO(3)$ invariant, one may compute the
product of the moment generating functions also after changing
from the monomial basis to the standard $SO(3)$-representations
basis.

Thus, we would like to compute the coefficients of the distribution
$\iota_{\BR^3}(\mu) \in L^2(\BR^3 \times \BS^2 \times \BS^2)$
by the decomposition of this space to $SO(3)$-irreducible representations.

\begin{rmr}[Embeding moment functionals on $L^2(\BR^3)$ in functionals on $\BR^+\times \BS^2$] There is an isomorphism
  \[ \bigoplus_{n \leq d,\text{ and } 2|n-d}r^d[n]^\vee\to
  \bigoplus_{d,n}r^d[n]^\vee,
  \]
  where we think on the first space as operators on $L^2(\BR^3)$ and on
  the second as operators on $L^2(\BR^+\times \BS^2)$.
\end{rmr}
The following lemma is a effective form of proposition-definition \ref{prd:d_s2cubed}.
\begin{lma}
There is an effective isomorphism
\[L^2(\BR^3) \hat\otimes L^2(\BS^2) \hat\otimes L^2(\BS^2) \cong L^2(SO(3)) \hat\otimes L^2(\mathbb{R}^+\times I_3).\]
Specifically, we consider the basis for dense set of functionals on
the left hand side above given by using the
spherical harmonics as the
basis for $L^2(S)^2$, and the isomorphism
\[
 (\bigoplus_{n \leq d} V_n)\otimes(\bigoplus_{n\leq d} [n]^\vee)\otimes(\bigoplus_{n\leq d} [n]^\vee) \cong
    \left(\bigoplus_{n \leq d,\text{ and } 2|n-d}r^d[n]^\vee\right)\otimes(\oplus_{n\leq d} [n]^\vee)^{\otimes^2}.
\]
Then, assuming a basis for dense set of functionals on
$L^2(\mathbb{R}^+ \times I_3)$ given by
$\{e_n\}_{n = 0}^{\infty}$,
any basis element of the space of functionals above may be approximated as well
as we need using a functional in
is
$(\bigoplus_{n \leq K} [n] \otimes [n]^\vee) \otimes \mathrm{Span} (e_i \mid i \leq K)$ for some $K$.
\end{lma}
\begin{proof}
  We want to get an effective form of the pullback to $L^2$ of the $SO(3)$
  fiber bundles embedding
  \[
  h:(I_3\sm D_1)\times SO(3)\to (\BS^2)^3.
  \]
  Note that given a basis $e_i$ for $L^2(I_3)$, one gets a concrete
  representation
  for $L^2(I_3\times SO(3))$ by tensoring with $L^2(SO(3))$ -- which
  is effectively given as the span of the $D^j_{m'm}$. On the other hand,
  one gets another basis for the same space --- as an $SO(3)$ bundle --- by
  tensoring the Hopf pullbacks of three $Y_{jm}s$.
  In order to decompose one basis in terms of the other, one has
  to compute all the integrals of the form
  \[
  \int_{I_3\sm D_1}\int_{SO(3)} e_n D^l_{k'k}(h^*(D^j_{m0}D^{j'}_{m'0}D^{j''}_{m''0}))^\dagger.
  \]
  One decomposes the term ``inside'' $h^*$ using the Kronecker
  product formula, and as the map is an $SO(3)$ fiber map.
\end{proof}
\begin{rmr}
  The integrals in the proof above have two useful properties:
  \begin{itemize}
  \item They are nontrivial only if the term inside the $h^*$ is
    $D^l_{k'k}$ times some other factor, on which $SO(3)$ acts trivially.
  \item Fixing $n, l$, varying $k,k'$, and assuming the
    $SO(3)$ invariant part ``inside'' $h^*$ is the same (by taking sums of
    such expressions), the value of the integral is the same.
  \end{itemize}
  While it is not immediately clear how to utilize the second property
  computationally (since one has to ``repack'' sums inside $h^*$), the
  computational benefit of
  the first property is clear: We are computing
  the integrals of a six dimensional family of basis elements in terms
  of another such family, but for each element only a three dimensional
  family is non trivial. Nevertheless, we inquire whether a more effective
  convolution algorithm can be found. See question \ref{q:better packing of the data}.

  Finally let us touch the point of constructing a basis to $L^2(I_3\sm D_1)$.
  The easiest way to do this is to diagonalize (in linear-algebraic sense)
  the $SO(3)$ invariant functions
  on $(\BS^2)^3$ (i.e.\ the copies of the $0$th representation in
  all the possible tensors of three different $Y_{nm}$). Yet,
  it seems likely that finding a basis in a more effective way
  might improve the algorithm.
\end{rmr}
Recall that $\psi \colon W_7 \setminus \Delta_{W_7} \to SO(3) \times \BR^+ \times I_3$ is a trivialization of the $SO(3)$-bundle $W_7 \setminus \Delta_{W_7}$. By remark \ref{R:l2fiber}, we can use it in order to convolve and obtain an $SO(3)$-invariant tempered distribution.

Let $\iota_{SO(3)}(f) = f *_{SO(3)} f$, where $f$ is a tempered distribution over $W_7 \setminus \Delta_{W_7}$, with $\psi$ acting as a trivialization map.
\begin{cor}
  Assuming that $M$ satisfies $\star\star$, the coefficients of
  $\delta_{\ti{M}_{2\times 4}}$ are
given by $\cup F^{\mu}_{d,d'}$. Moreover, the distribution is continuous in
the coefficients.
\end{cor}
\begin{proof}
Note that the support of $\iota_{SO(3)}(\delta_{M_{2\times 2}})$
is $\ti{M}_{2 \times 4}$, since we defined $\iota_{SO(3)}$ using $\psi$.
By proposition \ref{p:m4}, the locus of double points
and degenerated points is of positive co-dimension. The continuity in
the coefficients follows from the continuity in the coefficients in
Riesz–Markov–Kakutani representation theorem as mentioned in
\ref{d:rep_theory_interpretation}.
\end{proof}
%
\section{Functional analytic aspects and  proof of the main theorem}\label{S:functional_analysis}
%
From the two previous sections we know that assuming $M$ satisfies
properties $\star,\star\star$, it may be recovered from
$\ti{M}_{2\times 4}$, and that the delta function of $\ti{M}_{2\times 4}$ may be
stably computed from its Fourier coefficients.
The object of this section is to wrap up the proof of the main theorem
by doing three things: First we will show that compact ``differentiable enough''
surfaces in $\BR^3$, away from a comeagre subset, satisfy properties
$\star,\star\star$. This is done by using Morse theory to prove the statements
for ``most'' polynomials, and concluding the claim by a density argument. This
tasks occupies most of this section.
Next, we show that the fibers from proposition \ref{p:m4} can be recovered from
the coefficient $F_{\gd_M,d,d'}$ where $d,d'<K$, up to accuracy
determined by $K$ and is continuous at the coefficients
finally we have to show that in this approximate fiber recovery we are able
to identify the singular points of the fibers, which are exactly where shifted
copies of $M$ intersect.
\begin{ntt}
  Denote by $\Pol_3(d)$ the space of degree $d$ polynomials in $\BR^3$.
  Using the metric of distance on the coefficients, $\Pol_3(d)$ is a complete
  metric space.
\end{ntt}
\begin{prd}\label{p:morse}
  Fix $R > 0$ positive. We denote by $B(R)$ the closed ball of radius $R$
  about the origin.

  Let $f\in\Pol_3(d)$ be a polynomial, such that the intersection of its null set with $B(R)$ (denoted by $Z_f$), is a smooth surface.
  Then there is an open neighborhood $f \in U_f\subset\Pol_3(d)$
  such that for all $\ti{f}$ in $U_f$,
  $Z_{\ti{f}}$ is also a smooth surface.
\end{prd}
\begin{proof}
  Define
  \[
  \begin{aligned}
    F:\Pol_3(d)\times\BR^3&\to\BR\\
    (g,p)&\mapsto g(p),
  \end{aligned}
  \]
  and pick a filtration $\{0\}=A_0\subset A_1\subset\cdots\Pol_3(d)$. Where
  $\dim A_n=n$. We induct on $n$, where at each
  step of the
  induction we find a closed neighborhood $C_{f,n}$ of $\{f\}$ in $\{f\}+A_n$ so that
  the null set of $F\restriction C_{f,0}\times B(R)$
  is diffeomorphic to $Z_f \times [-1, 1]^n$.
  For $n = 0$, by the hypothesis of the lemma, $f$ is a smooth
 function whose null set is a smooth surface.
 The induction step follows from Morse theory. At each
 step of the induction we need to pick the closed neighborhood so that
 $F\restriction C_{f,n}\times\BR^3$ misses the singular locus of $F$ (which is of
 positive co-dimension).

 In more detail, let $g \in A_{n+1} \sm A_{n}$.
 For each point $\tilde{f} \in C_{f,n}$ and $y \in B(R)$
 such that $\tilde{f}(y) = 0$, there is some $\epsilon > 0$
 such that the gradient of $\tilde{f} + \delta \cdot g$ at $y$ is
 non-zero for all $|\delta| < \epsilon$. Using compactness,
 we can find a single $\epsilon$ that works for all
 $\tilde{f} \in C_{f,n}$ and $y \in B(R)$.
 Let $C_{f, n + 1} = C_{f,n} + [-\frac{1}{2}\epsilon, \frac{1}{2}\epsilon] \cdot g$.
\end{proof}
\begin{lma}\label{L:deg5surjection}
  For any integer $d\geq 5$, and four points $p_0,\ldots, p_3\in\BR^3$ in general position,
  the linear map
  \[
  \begin{aligned}
    \Pol_3(d)&\to\oplus_{i=0}^3 \Pol_3(2)\\
    f(v)&\mapsto\text{the degree $2$ part at $0$ of }f(v-p_i)
  \end{aligned}
  \]
  is surjective.
\end{lma}
\begin{proof}
  Since the map is equivariant under the action the affine linear group,
  without loss of generality we may assume that
  \[
  p_0=(0,0,0),\quad p_1=(1, 0, 0),\quad p_2=(0, 1, 0),\quad p_3=(0, 0, 1).
  \]
  We now consider the following four subspaces of $Pol_3(d)$:
  \[
  W_0:=\Pol_3(2),\quad W_i:=x_i^3\cdot\Pol_3(2)\text{ for }i=1,2,3.
  \]
  The the following properties hold:
  \begin{itemize}
    \item The intersection of these spaces is the zero polynomial.
    \item For each $i$, the map $W_i\to \Pol_3(2)$ taking $f(v)$ to the degree
      $2$ part at $0$ of $f(v-p_i)$ is surjective.
    \item For each $i>0$, $j\neq i$, the map taking $f(v)$ to the degree part
      $2$ at $0$ of $f(v-p_j)$ is identically zero.
  \end{itemize}
  The claim follows.
\end{proof}
\begin{lma}\label{l:pol_approx}
  Let $f\in\Pol_3(d)$ be a generic polynomial such that
  the dimension of the null set of $f$ is two dimensional. Then, assuming
  $d \geq 5$, the null set of $f$ satisfies property $\star$.
\end{lma}
\begin{proof}
  Consider the space $\Pol_3(d)\times\BR^3\times \BS^2\times \BR^3$.
  For each $(p,n,g)\in\BR^3\times \BS^2 \times(\BR^3\sm\{0\})$, we compute the
  dimension
  of the varieties:
  \[
  \begin{aligned}
  A_{p,n,g}:=&\{f\in\Pol_3(d)|f(p)=f(p+g)=0, \nabla f(p),\nabla f(p+g)||n\},\\
  A'_{p,n,g}:=&\{ f\in A_{p, n, g}|\nabla f(p)=0\text{ or }\nabla f(p+g)=0 \text{ or}\\
  &\min(\rank S_p, \rank S_{p+g}, \rank (S_p-S_{p+g})) < 2\}.
  \end{aligned}
  \]
  By lemma \ref{L:deg5surjection}, the co-dimension of $A_{p, n, g}$ in $\Pol_3(d)$
  is $2\cdot(1 + 2)=6$, and the co-dimension of $A'_{p,n,g}$ in
  $A_{p,n,g}$ is at least $1$.\footnote{the ``at least'' here is due to $\BR$ not being
  algebraically closed - if we were to extend scalars to $\BC$, this would
  simply be $1$}
  Hence, denoting
  \[
  \begin{aligned}
    A:=&\{(f, p, n, g)\in\Pol_3(d)\times\BR^3\times \BS^2\times \BR^3\mid \\
       & f(p)=f(p+g)=0, \nabla f(p),\nabla f(p+g)\parallel n\},\\
    A':=&\{(f,p,n,g)\in A \mid \nabla f(p)=0\text{ or }\nabla f(p+g)=0 \text{ or}\\ &\min(\rank S_p, \rank S_{p+g}, \rank (S_p-S_{p+g})) < 2\},
  \end{aligned}
  \]
  we see that since $d > 1$ the co-dimension of $A$ in
  $\Pol_3(d)\times\BR^3\times \BS^2\times \BR^3$ is
  $2\cdot(1 + 2)=6$, and since $d > 2$ the co-dimension of $A'$ in
  $A$ is at least $1$.

  We now claim that the last co-dimension is equal to $1$ and not simply bounded
  by it: since the polynomial $f$ is of rank $>2$, the second derivatives
  of $f$ are non-constant everywhere. Hence, by the implicit function theorem,
  the second derivatives of the surface defined by
  $f$ are nowhere constant (they are the inverse of the matrix of polynomials).
  Thus,
  \[\det S_p \cdot \det S_{p + g} \cdot \det (S_{p+g} - S_p)\]
  is a (high degree) non-constant polynomial.

  Finally, given $f\in\Pol_3(d)$ such that
  the null set of $f$ in $\BR^3$ is a surface, take an open neighborhood
  $U_f\subset\Pol_3(d)$ as in proposition-definition \ref{p:morse}.
  Then the same co-dimension properties will
  hold for a fiber in $\BR^3\times \BS^2\times \BR^3$ over a generic element in
  $U_f$. However, these co-dimension properties are precisely $\star$.
\end{proof}
\begin{lma}\label{l:pol_approx2}
  Let $f\in\Pol_3(d)$ be a generic polynomial such that
  the dimension of the null set of $f$ is $2$ dimensional. Then, assuming
  $d > 2$, the null set $f$ satisfies property $\star\star$.
\end{lma}
\begin{proof}
  The proof is analogous to the proof of lemma \ref{l:pol_approx}.
  Consider the space
  \[\Pol_3(d)\times(\BR^3)^4\times(\BS^2)^4\times(\BR^3)\times SO(3).
  \]
  For each
  \[
  (p_1,\ldots,p_4,n_1,\ldots,n_4,d_1,d_2,g)\in(\BR^3)^4\times(\BS^2)^4\times(\BR^3\sm\{0\})^2\times(SO(3)\sm\mathrm{Id}),
  \]
  so that $g(n_1)=n_3, g(n_2)=n_4, g(d_1)=d_2$ we compute the dimension of the
  varieties:
  \[
  \begin{aligned}
    A_{p_1,\ldots,p_4,n_1,\ldots,n_4,d_0,d_1,g}:=&
    \{f\in\Pol_3(d)|f(p_i)=0, \nabla f(p_i)=n_i,\\p_2-p_1=d_1, &p_4-p_3=d_2,
    \nabla f(p_1) ||n_1, \nabla f(p_2)||n_2\},\\
    A'_{p_1,\ldots,p_4,n_1,\ldots,n_4,d_1,d_2,g}:=&\{f\in A_{p_1,\ldots,p_4,n_1,\ldots,n_4,d_1,d_2,g}, \\
    \exists i:\nabla f(p_i)=0&, \text{ or }\\
    \min(\rank (g_* S_{p_1}-&S_{g(p_3)}),\rank (g_* S_{p_2}-S_{g(p_4)})) < 2, \text{ or }\\
     \min(\rank d \iota_4 \restriction & T_{(d_1, n_1, n_2)} M_{2 \times 2}(f), \\
    \rank d \iota_4 \restriction & T_{(d_2, n_3, n_4)} M_{2 \times 2}(f)) < 3 \}
  \end{aligned}
  \]
  Where $M_{2 \times 2}^f$ is the surface of all $(m_2 - m_1, n_1, n_2)$ such that $f(m_1) = f(m_2) = 0$ and $\nabla f(m_1) \parallel n_1, \nabla f(m_2) \parallel n_2$.

  In order to be able to work with $M_{2 \times 2}(F)$, let us note that for a
  neighborhood of $m_1, m_2$ in which the shape operators and their difference
  are both of full rank, we can locally parameterize $M_{2 \times 2}(F)$ by
  $(G^{-1}(n_2) - G^{-1}(n_1), n_1, n_2)$ where $G$ is the Gauss map restricted to
  a neighborhood around a point $m$ in which it is open.

  By the definition, we can compute the restriction of the linear map
  $d \iota_4$ to the tangent space of $M_{2 \times 2}(F)$ at each point, using the
  first two derivatives. The rest of the proof is completely identical to the
  proof of lemma \ref{l:pol_approx}.
\end{proof}
Recalling that we identify a surface $M \subseteq \mathbb{R}^3$ with
$\delta_M \in \mathcal{D}$, there are two points to consider when moving from
null sets of polynomials to general surfaces:
\begin{prp} If a continuous function with a compact null set is approximated, then so is the null set.
\end{prp}
\begin{thm}[See \cite{Na}]
  Let $A$ be a subalgebra of the algebra $C^\infty(X)$ of
  smooth functions on a finite dimensional smooth manifold $X$. Suppose that
  $A$ separates the points of $X$ as well as tangent vectors of
  $M$, then $A$ is dense in $C^\infty(X)$.
\end{thm}
Utilizing the theorem and proposition below, we conclude the first objective
of this section:
\begin{cor}\label{c:co_meagre}
  Properties $\star$
  and $\star\star$ hold for a co-meagre subset of compact
  $C^3$-orientable surfaces, in the subspace topology induced from $\mathcal{D}$.
\end{cor}
\begin{proof}
  First, note that both $\star$ and $\star\star$ may be phrased in terms
  of the null sets of some polynomials in the first and second derivatives
  of the surface.

  In particular, for a $C^3$-orientable surface, the sets $X_1'$ and $X'_3$ (which, in
  general, might not be of dimension 1 and 3 respectively), are composed of
  sub-manifolds of $X_2$ and $X_4$. Therefore, saying that they have the
  correct dimension is equivalent to showing that they are of measure zero.

  We claim that for each $\epsilon > 0$, the set of all $C^3$-manifolds in
  which the area of $X_1'$ is $<\epsilon$ is open (and similarly for $X_3'$).
  This is obvious by the continuity of the second derivatives with respect to
  the chosen topology and the closure of the condition in the definitions of
  $X_1'$ and $X_3'$.

  By lemma \ref{l:pol_approx} and lemma \ref{l:pol_approx2}, those sets are
  dense. By taking $\epsilon = \frac{1}{n}$, we conclude that the set of all
  manifolds that satisfy $\star$ and $\star\star$ includes a countable
  intersection of dense open sets, and therefore it is comeagre.
\end{proof}
The next sequence of claims establish the fact that one
can recover an approximation of a fiber of a manifold using
the moments of the manifold. This is the reconstruction step
in the proof of propositions \ref{p:m2} and \ref{p:m4}. Using
this, we can finally connect all the dots and finish the proof
of the main theorem.
\begin{prp}\label{P:pol_density}
  As monomials span a dense subspace of $L^2_{cs}(\BR^3)$, and
  spherical harmonics functions span a dense subspace of $L^2(\BS^2)$,
  products of monomials and spherical harmonic functions span a dense
  subspace of $L^2_{cs}(\BR^3\times \BS^2)$.
\end{prp}
\begin{cor}\label{c:integral_approx}
  by proposition \ref{P:pol_density} we may approximate the integrals of the
  function against $\gd_{\ti{M}_{2\times4}}$ on any ball
  using the values of the coefficients
  $F_{\gd_M,d,d'}$ for sufficiently large $d,d'$.
\end{cor}
\begin{prp}\label{p:ball_integral}
  Given a small enough ball -- in a $d$ dimensional ambient space --
  about a point $p$ on a smooth $d'$ dimensional manifold, the integrals
  of the delta function of the manifold on the ball is the volume of a $d'$
  dimensional ball in $\BR^n$. Hence, if two such manifolds pass through $p$,
  the integral on the ball of the sum of the delta functions is twice
  this volume.
\end{prp}
\begin{proof}[Proof of the main Theoem]
  By theorem \ref{t:con_moments} and
  proposition \ref{P:l2fiber}, the coefficients $F_{\gd_M,d,d'}$ are the Fourier
  Coefficients of $\gd_{\ti{M}_{2\times m}}$, and by corollary \ref{c:co_meagre},
  we may assume $M$ satisfies $\star$ and $\star\star$.
  Following proposition \ref{p:m4}, pick some
  $f\in\BS^2\times\BR^+\times(I_3\sm\ D_1)$ and cover
  the fiber $\BR^3\times \BS^2\times\{f\}$ with
  balls of radius $\gep$ which depends on the bound of the sectional
  curvature (the balls are not in the fiber - they are in the
  ambient space).
  Given $F_{\gd_M,d,d'}$, possibly known only up to some error $\gep'$, for sufficiently
  large $d,d'$, we use corollary \ref{c:integral_approx} to compute the
  intersection of $\ti{M}_{2\times 4}$ with balls which are with
  distance from the origin up to four times the bounding ball of $M$,
  with accuracy which depends on $d,d',\gep'$.

  By proposition \ref{p:ball_integral} we know through which of these balls
  the restriction of $\ti{M}_{2\times4}$ passes, and through which ones it has
  a double point. By proposition \ref{p:m4} double points are the worst form of
  singularity which we will encounter (for generic $f$), and removing them,
  we remain -- again by proposition \ref{p:m4} with a set of copies of $M$
  minus the isolated removed points, up to rotations and translations,
  and up to an error which depends on $\gep$. Since $M$ is connected and
  compact, for sufficiently small $\gep$, we can separated the different copies
  of $M$ from each other.
\end{proof}
\section{Open Problems}
We conclude this paper with a list of open problems.

\begin{question}
Are there two smooth manifolds with the same 4-degree invariants which are not in the same $SE(3)$-orbit?
Similarly, are there two separated sequences of smooth manifolds such that their invariants converge pointwise to the same value?
\end{question}
By our main result, this situation cannot occur in manifolds that satisfy $(\star\star)$.

In the definition of properties $(\star)$ and $(\star\star)$, the first requirement
is designed to make the reconstruction using fiber selection work, in the sense that most
fibers correspond to finitely many copies of the manifold.
\begin{question}
Can we weaken the requirements of $(\star), (\star\star)$, potentially making the corresponding collection of bad points $X_1', X_3'$ smaller?
In particular, can we replace the first requirements with the conclusions of lemmas \ref{L:prop_star1_is_regualrity} and \ref{lemma:starstar1_equiv_iota_rank} respectively (i.e.\ the local injectivity of the projection from $X_2, X_4$ to the manifold)?
\end{question}
\begin{question}\label{q:better packing of the data}
Is there a better way to pick the fiber in the reconstruction algorithm (instead of
picking it randomly)?
In particular, is there an effective way to select a fiber for which the copies of $M$ are
most separated? Can one achieve a better algorithm by considering several fibers together ?
\end{question}

In section \ref{S:rep}, we translated our geometrical algorithm into a
computation based on the coefficients of the delta function of the manifold
(in the generation of the invariants and in the reconstruction).
Our choice of basis makes both computations heavy.
\begin{question}
Is there a better choice of
representation, in the functional analytic side, that can reduce the
computational cost of the algorithm, and produce a denser representation of
the invariants?
\end{question}
We are using the functional analytic tools in a very weak sense.
In particular, while concentrating on a single fiber we effectively loss almost all
information from the other parts of the distribution.
\begin{question}
Is there a more efficient way to use the information from the invariants
in order to reconstruct the manifold?
\end{question}

A positive answer to the following question will likely require a different, less geometrical, approach.
\begin{question}
Is there a way to effectively reconstruct the surface using a collection of 3-degree invariants, or even 2-degree invariants?
\end{question}
\begin{question}
In the proof of the main theorem, we stated that for every error $\epsilon$ in the coefficients of the reconstructed manifold, there is a sufficiently small $\epsilon'$ and sufficiently large $d,d'$ that guarantee this level of accuracy. What is the connection between $\epsilon$ and $\epsilon', d, d'$?
\end{question}

\end{document}